\documentclass[12pt]{amsart}
\usepackage{latexsym}
\usepackage{amssymb}
\usepackage{amsmath}

\newtheorem{Thm}[subsection]{Theorem}
\newtheorem{Lemma}[subsection]{Lemma}
\newtheorem{Rem}[subsection]{Remark}

\newtheorem{Prop}[subsection]{Proposition}
\numberwithin{equation}{section}% equation numbers within sections
\newcommand{\ben}{\begin{enumerate}}
\newcommand{\een}{\end{enumerate}}

\newcommand{\bec}{\begin{center}}
\newcommand{\eec}{\end{center}}

\newcommand{\beq}{\begin{equation}}
\newcommand{\eeq}{\end{equation}}

\newcommand{\bdm}{\begin{displaymath}}
\newcommand{\edm}{\end{displaymath}}

\newcommand{\R}{\mathbb{R}}

\newenvironment{proofof}[1]{{\sc Proof of #1}}{\quad\lower0.05cm\hbox{$\square$}\medskip}

\newcommand{\bgp}{\bigskip}

\title[Gauge for Fractional Laplacian Schr\"{o}dinger Operators]
{Existence of the Gauge for Fractional Laplacian Schr\"{o}dinger Operators}

\author{Michael W. Frazier}
\address{Mathematics Department, University of Tennessee, Knoxville,
Tennessee 37922} \email{mfrazie3@utk.edu}

\author{Igor E. Verbitsky}
\address{Department of Mathematics, University of Missouri, Columbia, Missouri 65211}
\email{verbitskyi@missouri.edu}



\begin{document}

\begin{abstract}
Let $\Omega \subseteq \R^n$ be an open set, where $n \geq 2$. Suppose $\omega $ is a locally finite Borel measure on $\Omega$.  For $\alpha \in (0,2)$, define the fractional Laplacian $(-\triangle )^{\alpha/2}$ via the Fourier transform on $\R^n$, and let $G $ be the corresponding Green's operator of order $\alpha$ on $\Omega$. Define  $T(u) = G(u \omega).$  If $\Vert T \Vert_{L^2(\omega) \rightarrow L^2 (\omega)} <1$, we obtain a representation for the unique weak solution $u$ in the homogeneous Sobolev space $L^{\alpha/2, 2}_0 (\Omega)$ of 
  \[ (-\triangle)^{\alpha/2} u = u \omega + \nu \,\,\, \mbox{on} \,\,\, \Omega, \,\,\,  u=0 \,\,\, \mbox{on} \,\,\,  \Omega^c, \]
for $\nu$ in the dual Sobolev space $L^{-\alpha/2, 2} (\Omega)$.  If $\Omega$ is a bounded $C^{1,1}$ domain, this representation yields matching exponential upper and lower pointwise estimates for the solution when $\nu = \chi_{\Omega}$.  These estimates are used to study the existence of a solution $u_1$ (called the ``gauge") of the integral equation $u_1=1+G(u_1 \omega)$ corresponding to the problem 
\[  (-\triangle)^{\alpha/2} u = u \omega  \,\,\, \mbox{on} \,\,\,  \Omega, \,\,\, u \geq 0 \,\,\, \mbox{on} \,\,\, \Omega, \,\,\, u=1 \,\,\, \mbox{on} \,\,\, \Omega^c . \]  
\noindent{We} show that if $\Vert T \Vert <1$, then $u_1$ always exists if $0<\alpha <1$.  For $1 \leq \alpha <2$,  a solution exists if the norm of $T$ is sufficiently small. We also show that the condition $\Vert T \Vert <1$ does not imply the existence of a solution if $1 < \alpha <2$.
\end{abstract}

\maketitle \vfill

\eject

\tableofcontents

\section{Introduction}

Suppose $\Omega \subseteq \R^n$, where $n \geq 2$, is a nonempty open set (possibly the whole space), $\omega$ is a locally finite (positive) Borel measure on $\Omega$, and $\alpha \in (0, 2)$.  We consider the problems:  

\begin{equation}\label{u0equation} 
\left\{ \begin{aligned}
(-\triangle )^{\alpha/2} u & = \omega u +\nu \, \,   &\mbox{in} \, \,  \Omega,  \\
u & = 0 \, \,& \mbox{on} \, \, \Omega^c,
\end{aligned}
\right.  
\end{equation} 
where $\Omega^c = \R^n \setminus \Omega$, and
\begin{equation}\label{u1equation} 
\left\{ \begin{aligned}
(-\triangle )^{\alpha/2} u & = \omega u \, \,&  &  \mbox{in} \, \,  \Omega, \quad  u \ge 0,  \\
u & = 1 \, \, &  &\mbox{on} \, \,  \Omega^c.
\end{aligned}
\right.  
\end{equation} 

Here the fractional Laplacian $(-\triangle)^{\alpha/2}$ is the non-local operator defined in $\R^n$ via the Fourier transform by $((- \triangle)^{\alpha/2} u)^{\, \hat{}}(\xi) = |\xi|^{\alpha} \hat{u} (\xi)$, when $|\xi|^{\alpha} \hat{u} (\xi) \in \mathcal{S}^{\, \prime} (\R^n)$ (for example when $u $ belongs to the Sobolev space $H^{\alpha} (\R^n)$). For sufficiently nice $u$ (e.g., $u$ in the Schwarz class), there is a pointwise representation
\begin{equation} \label{pointwisefracLap}
 (- \triangle)^{\alpha/2} h (x) = C(\alpha, n) \,\, p.v. \int_{\R^n} \frac{h(x) - h(y)}{|x-y|^{n +\alpha}} \, dy , 
\end{equation}
see \cite{DPV}, Proposition 3.3.  See \cite{Kw} for equivalent definitions on $\R^n$. There are other interpretations of the fractional Laplacian restricted to a domain; see for example \cite{V} and \cite{MN}, where the interpretation above is called the restricted fractional Laplacian $(-\triangle_{|\Omega})^{\alpha/2}$ (\cite{V}) or the Dirichlet fractional Laplacian $(-\triangle_{\Omega})_D^{\alpha/2}$ (\cite{MN}). Problems (\ref{u0equation}) with $\nu= \chi_{\Omega}$ and (\ref{u1equation}) were considered for the classical Laplacian ($\alpha =2$) in \cite{FV}.  

The (minimal) solution $u=u_1$ to (\ref{u1equation}) 
is called \textit{the gauge function} (the \textit{Feynman-Kac gauge}, or simply the \textit{gauge}) in the probability literature. 
In the case $d \omega =q(x) \, dx$, $q \in L^1_{loc}(\Omega)$, 
it can be expressed in the form 
\[
u_1(x) = E^x\left(e^{\int_0^{\tau_\Omega} q(X_s) ds}\right), 
\]
where $X_t$ is the (scaled) Brownian motion if $\alpha=2$, or a symmetric $\alpha$-stable L\'{e}vy process starting at $x$ if $0<\alpha<2$, and $\tau_\Omega = \inf\{t > 0: \, X_t \in \Omega^c\}$. If $q$ lies in the corresponding Kato class, then the 
so-called Gauge Theorem says (see \cite{CZ}, Theorem 4.19, if $\alpha=2$, 
and \cite{BBK}, Theorem 2.9, if $0<\alpha<2$) that,  for a bounded domain $\Omega$, $u_1$ is uniformly bounded, provided 
$(\Omega, q)$ is ``gaugeable'', which is equivalent to the condition  
$||T||<1$ discussed below for general $\omega \ge 0$. 

We do not impose any conditions of Kato type on $\omega\ge 0$, and 
consequently in this general setup the gauge 
is no longer uniformly bounded. As we will show below, it is finite a.e. only under additional conditions 
in the case $1<\alpha\le 2$.

Let $G(x,y) = G^{(\alpha)}(x,y)$ be the Green's function for $(-\triangle)^{\alpha/2}$ on the domain $\Omega$, defined as in \cite{L}, Ch. IV.5. Then $G$ is a non-negative, symmetric function on $\R^n \times \R^n$.  We also denote by $G= G^{(\alpha)}$ the corresponding Green's operator of order $\alpha$, acting on a measure $\mu $ on $\Omega$ by
\begin{equation}
  G \mu (x) = \int_{\Omega} G(x,y) d \mu (y),  \,\,\, x \in \R^n. \label{defngreenpot}
\end{equation}
If $d \mu = f \, dx$, we write $Gf $ instead of $G \mu$. 

For $f \in C_0^{\infty} (\Omega)$, the equation $(-\triangle)^{\alpha/2} u =f$ on $\Omega$, with $u=0$ on $\Omega^c$, has solution $u = G f$.  By applying $G$ to both sides of the equation $(-\triangle )^{\alpha/2} u = \omega u +\nu$, we obtain the corresponding integral equation 
\begin{equation}\label{intformulation}
u = G(\omega u) + G\nu.
\end{equation}

Let $T$ be the operator  
\begin{equation} \label{defT}
Tf (x) = G(f \omega ) (x) = \int_{\Omega} G(x,y) f(y)\,  d \omega (y).
\end{equation}
Then (\ref{intformulation}) becomes $u=Tu + G\nu$, which has the formal solution 
\[   u = (I-T)^{-1} G\nu= \sum_{j=0}^{\infty} T^j G\nu  .\]
Our main assumption is that $\Vert T \Vert \equiv \Vert T \Vert_{L^2 (\omega) \rightarrow L^2 (\omega) } <1$.  This assumption is equivalent (see Lemma \ref{Tbetaequiv} below) to the existence of some $\beta <1$ such that 
\begin{equation}\label{embedd}
\Vert h \Vert_{L^2 (\omega)} \leq \beta \Vert (-\triangle)^{\alpha/2} h \Vert_{L^2 (\R^n)}  , \,\,\, \mbox{for all} \,\,\, h \in C^{\infty}_0 (\Omega).
\end{equation}

We inductively define kernels $G_j (x,y)$ on $\Omega \times \Omega$ for $j \geq 1$ by setting $G_1 = G$ and, for $j \geq 2$,
\begin{equation}
 G_j (x,y) = \int_{\Omega} G_{j-1} (x,z) G(z,y) \, d \omega (z). \label{defGj}
\end{equation}
Then $T^j f(x) = \int_{\Omega} G_j (x,y) f(y) \, d \omega (y)$ for $j \geq 1$, by an application of Fubini's theorem.  We define the fractional Green's function $\mathcal{G}$ of order $\alpha$ associated to $\omega$ and $\Omega$: 
\begin{equation}
 \mathcal{G} (x,y) = \sum_{j=1}^{\infty} G_j (x,y).  \label{defGreenfcnSchr}
 \end{equation}
and the corresponding operator $\mathcal{G} \nu (x) = \int_{\Omega} \mathcal{G} (x,y) \, d\nu (y)$.  Note that each $G_j$, and hence $\mathcal{G}$, is symmetric and nonnegative. Another use of Fubini's theorem gives $T^j G\nu (x) = \int_{\Omega} G_{j+1} (x,y) \, d\nu (y)$.  We give the name $u_0$ to our formal solution above, and note that
\begin{equation} \label{defu0}
 u_0 (x) = \sum_{j=0}^{\infty} T^j G \nu (x) = \sum_{j=1}^{\infty} \int_{\Omega} G_j (x,y) \, d\nu(y) = \mathcal{G} \nu (x). 
 \end{equation} 
If $\nu$ is a positive measure, then $u_0$ satisfies $u = Tu + G\nu$ at every point, in the sense of functions with possibly infinite values. We say $u_0$ is a \textit{pointwise} solution of (\ref{intformulation}) if $u_0 < \infty$ Lebesgue-a.e. on $\Omega$.  

The homogeneous Sobolev space  $L^{\alpha/2,2}_0 (\Omega)$ is defined, for $\alpha \in (0, 2)$, to be the closure of $C_0^{\infty} (\Omega)$ with respect to the norm 
\begin{equation}
\Vert u \Vert_{L^{\alpha/2,2}_0 (\Omega)} = \Vert (-\triangle)^{ \alpha/4} u \Vert_{L^2 (\R^n)} .  \label{defSob}
\end{equation} 
We also define $L^{-\alpha/2,2} (\Omega)$ to be the dual of $L^{\alpha/2,2}_0 (\Omega)$.

If $\Vert T \Vert <1$, then, by the Lax-Milgram Theorem, there is a unique weak solution $u \in L^{\alpha/2,2}_0 (\Omega)$ to (\ref{u0equation}) for each $\nu \in L^{- \alpha/2, 2} (\Omega)$ (see $\S 2$).  The following proposition shows that this weak solution is realized via $\mathcal{G}$, and that the condition $\Vert T \Vert <1$ on $\omega$ is close to being necessary for the existence of a solution.

\begin{Prop}\label{u0thm} Suppose $n \geq 2$ and $0<\alpha <2$.  Let $\Omega \subseteq \R^n$ be open and let $\omega$ be a (positive) Borel measure on $\Omega$. 

\vspace{0.1in}

(A) Suppose $\Vert T \Vert < 1$.  If $\nu \in L^{- \alpha/2, 2} (\Omega)$ is a (positive) measure,  
then $\mathcal{G} \nu$ is a non-negative pointwise solution to (\ref{intformulation}), and $\mathcal{G} \nu \in   L^{\alpha/2,2}_0 (\Omega)$, with 
\begin{equation} \label{normu0est}
 \Vert \mathcal{G} \nu \Vert_{L^{\alpha/2, 2}_0 (\Omega)} \leq \frac{1}{1- \Vert T \Vert}  \Vert \nu \Vert_{L^{-\alpha/2, 2} (\Omega)} . 
\end{equation}

Also, $\mathcal{G}$ extends to be a bounded operator from $L^{- \alpha/2, 2} (\Omega)$ to $L^{\alpha/2, 2}_0 (\Omega)$ with norm at most $(1-\Vert T \Vert)^{-1}$. 
For a general $\nu \in L^{-\alpha/2, 2} (\Omega)$, $\mathcal{G} \nu$ is the weak solution to (\ref{u0equation}).

\vspace{0.1in}

(B) If (\ref{intformulation}) has a non-trivial non-negative pointwise solution $u$ for some positive measure $\nu$, then $\Vert T \Vert \leq 1$.  If also $u \in L^{\alpha/2,2}_0 (\Omega)$, then $\nu \in L^{-\alpha/2,2} (\Omega)$.
\end{Prop}

\vspace{0.1in}

{\bf Remark.}  For $\alpha =2$, the results of Proposition \ref{u0thm} hold for $n \geq 3$ by the same methods.  For $n=2$, they hold for domains $\Omega$ with a non-trivial  non-negative Green's function.  See also \cite{DD}, \cite{FV} and the references given there for $d \omega =q \, dx$, where $q \in L^1_{loc}(\Omega)$, in bounded 
smooth domains $\Omega$. 

\vspace{0.1in}

The existence of the solution operator $\mathcal{G}$ in Proposition \ref{u0thm} and its mapping properties follow easily from the Lax-Milgram Theorem.  The specific representation of $\mathcal{G}$ and (ii) in Theorem \ref{u0thm} may be new.  This representation will allow us to use the results of \cite{FNV} to obtain pointwise estimates for the solution to (\ref{u0equation}) in the case where $\nu  = \chi_{\Omega}$, which we will then relate to equation (\ref{u1equation}).  

Although Proposition \ref{u0thm} holds for a general open set $\Omega$, for further conclusions we require some additional conditions.  First, we must have $\chi_{\Omega} \in L^{-\alpha/2,2} (\Omega)$, This condition holds whenever $|\Omega | < \infty$.  This follows from (\ref{G1equiv}) and Lemma \ref{Gbnded} below, because (\ref{G1equiv}) shows that $\int_{\Omega} G \chi_{\Omega} \, dx < \infty$. Alternately, note that $\chi_{\Omega} \in L^{q^*} (\R^n)$ for any $q^*$, in particular for $q^* = 2n/(n+\alpha)$, the conjugate index to $p^* = 2n/(n-\alpha)$.  By the Sobolev imbedding theorem,  $L^{\alpha/2, 2}_0 (\Omega)$ imbeds in $L^{p^*}$ continuously (see e.g., \cite{DPV}, Theorem 6.5), so $L^{q^*}$ imbeds continuously in the dual space $L^{-\alpha/2, 2} (\Omega)$.  

In addition, to apply \cite{FNV}, we require the Green's function of $\Omega$ to have certain properties (see $\S$3).  These properties hold if $\Omega$ is a bounded $C^{1,1}$ domain.  

\begin{Thm}\label{nu=1thm} Suppose $n \geq 2$ and $0< \alpha < 2$.  Let $\Omega \subset \R^n$ be a bounded $C^{1,1}$ domain and let $\omega$ be a (positive) Borel measure on $\Omega$. 

\vspace{0.1in}

(i) Suppose $\Vert T \Vert <1$.  For $x \in \Omega$, let $m(x) = \delta(x)^{\alpha/2}$, where $\delta (x)$ is the distance from $x$ to the boundary of $\Omega$.  Let $u_0= \mathcal{G} \chi_{\Omega}$ be the solution of (\ref{u0equation}) when $d\nu = \chi_{\Omega} \, dx$.  Then there exist constants $C= C(\Omega, \alpha, \Vert T \Vert)>0$, and $C_1= C_1 (\Omega, \alpha)>0$, such that 
\begin{equation}\label{u0upperestimate}
u_0 (x) \leq C_1 m (x) e^{C \frac{ Tm(x)}{m (x)}}, \,\,\, \mbox{for all} \,\,\, x \in \Omega.
\end{equation} 

\vspace{0.1in}

(ii) Conversely, if $u$ is any non-negative solution of (\ref{intformulation}) with $d\nu= \chi_{\Omega} \, dx$, then there exist positive constants $c= c(\Omega, \alpha)$ and $c_1 = c_1 (\Omega, \alpha)$ such that
\begin{equation}\label{ulowerestimate}
 u (x) \geq c_1 m (x) e^{c\frac{ Tm(x)}{m (x)}}, \,\,\, \mbox{for} \,\, a.e. \,\,\, x \in \Omega.
\end{equation}   
 
\end{Thm}

In fact, in statement (ii), an estimate with more precise constants is proved in \cite{GV}: any non-negative solution $u$ satisfies the lower bound 
\[ u (x) \geq  s (x) e^{\frac{ Ts(x)}{s (x)}}, \,\,\, \mbox{for} \,\, a.e. \,\,\, x \in \Omega,\]
where $s(x)= G \chi_\Omega(x)\approx m(x)$.

For $d\nu = \chi_{\Omega}\, dx$, our conclusions in Proposition \ref{u0thm} and Theorem \ref{nu=1thm} are very similar to the results for $\alpha =2$ in \cite{FV}. On the other hand, our conclusions for the solution of (\ref{u1equation}) are very different from those for $\alpha =2$ in \cite{FV}. To formulate the integral equation corresponding to (\ref{u1equation}), let $v=u-1$.  Then $v$ satisfies $(-\triangle)^{\alpha/2} v =  v \omega+ \omega$ on $\Omega$, with $v=0$ on $\Omega^c$.  Applying $G$ to both sides gives $v= G(v \omega) + G\omega = Tv + G \omega$.  Therefore we consider $v= (I-T)^{-1} G \omega = \sum_{j=1}^{\infty} T^j G \omega = \mathcal{G} \omega$.  Thus the integral equation analogue of (\ref{u1equation}) is  
\begin{equation} \label{intformulationu1}
u= 1 + G(u \omega)
\end{equation}
with solution (at every point, but with values that might be $+\infty$)
\begin{equation} \label{defu1}
u_1 (x) = 1 + \mathcal{G} \omega (x) = 1 + \int_{\Omega} \mathcal{G} (x,y) \, d \omega (y),
\end{equation}
for $x \in \Omega$, and $u_1 =1$ on $\Omega^c$.  We say that $u_1$ is a pointwise solution of $u=1 + G (u \omega)$ if $u_1 < \infty$ Lebesgue-a.e. on $\Omega$.  

For $x \in \Omega$ and $y \in \Omega^c$, let $P(x,y) = P^{\alpha} (x,y)$ be the Poisson kernel of order $\alpha$ for $\Omega$ (see \S 3).

\begin{Thm}\label{u1thm} Suppose $n \geq 2$.  Let $\Omega \subset \R^n$ be a bounded $C^{1,1}$ domain and let $\omega$ be a (positive) Borel measure on $\Omega$.

\vspace{0.1in}

(i) Suppose $0 < \alpha <1$ and $\Vert T \Vert <1$.  Then $u_1 \in L^1 (\Omega, dx)$ and hence $u_1$ is a pointwise solution of (\ref{intformulationu1}).  
Moreover, $u_1-1 \in L^{\alpha/2, 2}_0 (\Omega)$.

(ii) Suppose $1 \leq \alpha <2$.  Then there exists a  constant $\gamma= \gamma (\alpha, \Omega)$  $\in (0, 1)$ such that if $\Vert T \Vert <\gamma$, then $u_1 \in L^1 (\Omega, dx)$ and hence $u_1$ is a solution of (\ref{intformulationu1}).

\vspace{0.1in}

When $0<\alpha <1$ and $\Vert T \Vert <1$, or $1 \leq \alpha <2$ and $\Vert T \Vert < \gamma$, for $\gamma$ as in (ii), there exist constants $C_3 (\Omega, \alpha), C_4 (\Omega, \alpha,  \Vert T \Vert), $ such that 
\begin{equation}\label{u1pointwiseupper}
u_1 (x) \leq C_3 \int_{\Omega^c} e^{C_4 \int_\Omega G(x,y) \frac{P(y,z)}{P(x,z)} d \omega(y)}  \,  P(x,z) \, dz.
\end{equation} 
Also, if $u$ is a non-negative solution to $u= 1 + G (u \omega)$, then there exist constants $c_3 (\Omega, \alpha)$ and $c_4(\Omega, \alpha)$ such that
\begin{equation}\label{u1pointwiselower}
u (x) \geq c_3 \int_{\Omega^c} e^{c_4 \int_\Omega G(x,y) \frac{P(y,z)}{P(x,z)} d \omega(y)}  \,  P(x,z) \, dz.
\end{equation}

\end{Thm}

The next result states that in the case $1< \alpha < 2$, the condition $\Vert T \Vert <1$ is not sufficient to obtain that $u_1$ is finite a.e. (as is the case when $\alpha =2$, by the results in \cite{FV}).  From the proof, we will also see that $\gamma \rightarrow 0$ as $\alpha \rightarrow 2^-$ in Theorem \ref{u1thm}.

\begin{Thm}\label{example}  Suppose $n \geq 2$.  Let $\Omega \subset \R^n$ be a bounded Lipschitz domain. Suppose $1 < \alpha <2$.  Then there exists a positive measure $\omega \in L^1_{loc} (\Omega)$ such that $\Vert T \Vert <1$, but $u_1$ is identically $+ \infty$ on $\Omega$.
\end{Thm}

In \cite{FV}, it was shown for $\alpha =2$ that $u_1$ is finite a.e. if  $\Vert T \Vert <1$ and  $\omega$ satisfies an additional boundary condition (the exponential integrability of the balayage of $\delta d \omega$), and conversely if there is a solution, then $\Vert T \Vert \leq 1$ and the balayage of $\delta \, d \omega$ is exponentially integrable with a different constant. No such boundary conditions on $\omega$ appear in the case $0<\alpha<2$.

\vspace{0.1in}

\noindent{{\bf Acknowledgment.}  We thank Fedor Nazarov for helpful conversations on the subject matter of this paper, especially with regard to Theorems \ref{u1thm} and \ref{example}.}

\bgp

\section{Proof of Proposition \ref{u0thm}}\label{proofofu0thm}

We start by summarizing some background.  For $0 < \alpha <n$, the function $k_{\alpha} (x) = c_{\alpha, n} |x|^{\alpha -n}$, where $c_{\alpha, n}$ is an appropriate normalization constant, has Fourier transform $\widehat{k_{\alpha}} (\xi) = |\xi|^{- \alpha}$.  On $\R^n$, the Riesz potential $I_{\alpha}$ of order $\alpha$ acts on a Borel measure $\mu$ by 
\begin{equation} \label{fracRiesz}
I_{\alpha} \mu (x) = k_{\alpha} \ast \mu (x) =  c_{\alpha, n} \int_{\R^n} \frac{d \mu (y)}{|x-y|^{n - \alpha}} .
\end{equation}
Thus $I_{\alpha}$ serves as the Green's operator $G^{(\alpha)}$ on $\R^n$, since $(-\triangle)^{\alpha/2} I_{\alpha} \varphi = \varphi$ for sufficiently nice functions $\varphi$.  Note that $I_{\alpha}$ is self-adjoint and satisfies the semi-group property:  $I_{\alpha + \beta} = I_{\alpha} I_{\beta}$ for $\alpha, \beta >0$ such that $\alpha + \beta <n$.  See e.g., \cite{AH}, $\S$ 1.2.2, for these facts.  For $0 < \alpha < n$, define the homogeneous Sobolev space 
\[ L^{\alpha/2, 2} (\R^n) = \{ u \in \mathcal{S}^{\, \prime} (\R^n): (-\triangle )^{\alpha/4} u \in L^2 (\R^n)\},  \]
with norm 
\begin{equation} \label{homogSobnorm}
\Vert u \Vert_{L^{\alpha/2, 2} (\R^n)} = \Vert (-\triangle)^{\alpha/4} u \Vert_{L^2 (\R^n)}.
\end{equation}
Note that each $u \in L^{\alpha/2, 2} (\R^n)$ can be written as $u = I_{\alpha/2} f$ for $f = (-\triangle )^{\alpha/4} u \in L^2 (\R^n)$, with $\Vert f \Vert_{L^2 (\R^n)} = \Vert u \Vert_{L^{\alpha/2, 2} (\R^n)}$.  

Also, define
\[ 
L^{-\alpha/2, 2} (\R^n) = \{ u \in \mathcal{S}^{\, \prime} (\R^n): |\xi|^{ -\alpha/2} \hat{u} (\xi) \in L^2 (\R^n) \},    
\]
with norm $\Vert u \Vert_{L^{-\alpha/2, 2} (\R^n)} = \Vert |\xi|^{-\alpha/2} \hat{u} (\xi) \Vert_{L^2 (\R^n)}$.  A Borel signed measure $\mu$ on $\R^n$ has finite $\mu$ on $\R^n$ has finite $\alpha$-energy $\int_{\R^n} I_{\alpha} \mu \, d \mu$ if and only if $\mu \in L^{-\alpha/2,2} (\R^n)$, and 
\begin{equation} \label{energy}
\int_{\R^n} I_{\alpha} \mu \, d \mu = \Vert I_{\alpha/2} \mu \Vert_{L^2 (\R^n)}^2 = \Vert \mu \Vert^2_{L^{-\alpha/2,2} (\R^n)} .
\end{equation}
For $\mu \in L^{- \alpha/2, 2} (\R^n)$, 
\begin{equation} \label{dualenergy}
\Vert I_{\alpha} \mu \Vert_{L^{\alpha/2, 2} (\R^n)} = \Vert (-\triangle )^{\alpha/4} I_{\alpha} u \Vert_{L^2 (\R^n)} =   \Vert \mu \Vert_{L^{-\alpha /2,2} (\R^n)}. 
\end{equation}
In fact, $I_{\alpha}$ maps  $L^{-\alpha/2,2} (\R^n)$ isometrically onto $L^{\alpha/2,2} (\R^n)$.

Let $L^{\alpha/2, 2 }_0 (\R^n)$ ($0 < \alpha < n$) be the homogeneous Sobolev space defined as the completion of $C^{\infty}_0 (\R^n)$ with respect to the $L^{\alpha/2, 2} (\R^n)$-norm. Then $L^{\alpha/2, 2}_0 (\Omega)$ is a closed subspace of $L^{\alpha/2, 2 }_0 (\R^n)$ with the inherited norm. Then $L^{-\alpha/2,2} (\R^n)$ is the dual space of $L^{\alpha/2,2}_0 (\R^n)$.  

For $E \subset \R^n$, we define the $\alpha$-capacity of $E$ relative to $\R^n$ by 
\[  \mbox{cap}_{\alpha} (E) = \inf \{ \Vert f \Vert_{L^2 (\R^n)}^2 : I_{\alpha/2} f \geq \chi_E, f \geq 0, f \in L^2 (\R^n) \}. \]
This capacity is sometimes denoted by $\mbox{cap}_{\alpha/2, 2} (\cdot)$, and is a 
special case of the ``non-linear'' capacity $\mbox{cap}_{s, p} (\cdot)$ defined by 
\[  \mbox{cap}_{s, p} (E) = \inf \{ \Vert f \Vert_{L^p (\R^n)}^p : I_ s f \geq \chi_E, f \geq 0, f \in L^p (\R^n) \}, \]
for $1<p<\infty$ and $0<s<n/p$. (See \cite{AH}, Sec. 2.3.)

A property holds {\em{quasi-everywhere}} (abbreviated q.e., or $\alpha$-q.e. if the value of $\alpha$ is not clear from context) if it holds except on a set of $\alpha$-capacity $0$. 

A function $f$ defined on an open set $\Omega\subset \R^n$ (or quasi-everywhere in $\R^n$) is said 
to be $\alpha$-quasicontinuous on $\Omega$ if, for every $\epsilon>0$, there exists an open set $G$ such that $\mbox{cap}_{\alpha} (G)< \epsilon$, 
and $f$ is continuous outside $G$. If $g \in L^2(\R^n)$ and $\alpha >0$ then $I_{\alpha/2} g $ is $\alpha$-quasicontinuous (in particular $I_{\alpha/2} g $ is finite $\alpha$-q.e.). Hence every function $f \in L^{\alpha/2, 2} (\R^n)$ has an $\alpha$-quasicontinuous representative $\tilde{f} = I_{\alpha/2} g$, for $g = (-\triangle )^{\alpha/4} f \in L^2 (\R^n)$. Also, if $\mu \in L^{-\alpha/2, 2} (\R^n)$ is a positive measure, then $I_{\alpha/2} \mu \in L^2 (\R^n)$, hence $I_{\alpha} \mu = I_{\alpha/2} (I_{\alpha/2} \mu) $ is $\alpha$-quasicontinuous.  Moreover, if two $\alpha$-quasicontinuous  functions coincide a.e., then they coincide $\alpha$-q.e. (see \cite{AH}, Sec. 6.1 for these facts about quasicontinuity in the case of the inhomogeneous Sobolev spaces; the proof for the homogeneous spaces is virtually the same). If $\{ f_j \}_{j=1}^{\infty} $ is a sequence of functions in $L^{\alpha/2, 2} (\R^n)  $ converging in norm to $f$, then there is a subsequence $f_{j_k}$ converging $\alpha$-q.e. to an $\alpha$-quasicontinuous representative $\tilde{f}$ of $f$ in $L^{\alpha/2, 2} (\R^n)  $. (This fact is a restatement of \cite{AH}, Proposition 2.3.8, by considering $g_j  \in L^2 (\R^n)$ such that $f_j = I_{\alpha/2} g_j$.)  In particular, if $f_j$ also converges to $f$ $\alpha$-q.e., then $f$ is $\alpha$-quasicontinuous.

Note that if $\lambda \in L^{- \alpha/2, 2} (\R^n)$ is a signed Borel measure, and $E \subset \R^n$ is a Borel set with $\mbox{cap}_{\alpha} (E)=0$, then $\lambda (E)=0$, as follows.  We may assume $\lambda$ is a positive measure. Let $\mathcal{A}$ denote the class of functions in the infimum defining $\mbox{cap}_{\alpha}$.  Then for all $f \in \mathcal{A}$,
\[  \lambda (E) \leq  \int_{\R^n} I_{\alpha/2} f \, d \lambda =   \int_{\R^n} f I_{\alpha/2} \lambda \, dx \leq   \Vert f \Vert_{L^2 (\R^n)} \Vert \lambda \Vert_{L^{- \alpha/2, 2} (\R^n)}  , \]
so taking the infimum over $f \in \mathcal{A}$ gives $\lambda (E)=0$. Observe that for $\lambda \in L^{- \alpha/2, 2} (\R^n)$ a signed measure and $f \in L^{\alpha/2, 2}_0 (\R^n)$, the quantity $\int_{\R^n} f \, d\lambda$ is not well-defined, because changing $f$ on a set of measure $0$ but positive $\alpha$-capacity may change the integral.  However, if $|\lambda| \in L^{- \alpha/2, 2} (\R^n)$ and $\tilde{f}$ an $\alpha$-quasicontinuous representative of the equivalence class of $f$ in $L^{\alpha/2, 2}_0 (\R^n)$, then $\int_{\R^n} \tilde{f} \, d\lambda$ is well-defined (i.e., independent of the choice of $\alpha$-quasicontinuous representative) and the duality pairing between $\lambda$ and $f$ is $\langle \lambda, f\rangle = \int_{\R^n} \tilde{f} \, d\lambda$ (see \cite{AH}, Sec. 7.1, equation (7.1.2)).

We also define the capacity of a compact set $E \subseteq \Omega$ relative to $\Omega$: 
\[  \mbox{cap}_{\alpha} (E, \Omega) = \inf \{ \Vert u \Vert_{L^{\alpha/2, 2}_0 (\Omega)}^2 : u \geq 1 \,\, \mbox{on} \,\, E, u \in C^{\infty}_0 (\Omega) \}.  \]
In the case $\Omega=\R^n$, we have $\mbox{cap}_{\alpha} (E, \R^n) =\mbox{cap}_{\alpha} (E)$ (\cite{AH}, Proposition 2.3.13).  For an open set $G\subset \R^n$, we set $\mbox{cap}_{\alpha} (G, \Omega) = \sup \{ \mbox{cap}_{\alpha} (E, \Omega): \, E \subset G, E \,\,\, \mbox{compact} \}$.    

The following is a dual form of Deny's Theorem (\cite{D}; also see Theorem 9.1.7 in \cite{AH}.) 

\begin{Lemma} \label{DenyCorollary}
Suppose $f \in  L^{\alpha/2, 2} (\R^n) $ and $f$ has an $\alpha$-quasicontinuous representative $\tilde{f}$ such that $\tilde{f} =0$ $\alpha$-q.e. on $\Omega^c$.  Then $f \in L^{\alpha/2, 2}_0 (\Omega)$.  
\end{Lemma}

\begin{proof}
Let $f$ be as in the assumptions and let $\mu \in L^{-\alpha/2, 2} (\R^n)$ be a (positive) measure with supp $\mu \subseteq \Omega^c$.  Then $\langle f, \mu \rangle = \int_{\R^n} \tilde{f} \, d\mu =0$, where here $\langle , \rangle$ denotes the pairing between $L^{\alpha/2, 2} (\R^n)$ and its dual $L^{-\alpha/2, 2} (\R^n)$.  By a theorem of Deny (\cite{D}, p. 143, or see \cite{AH}, Corollary 9.1.7 and the remarks in Section 9.13) any distribution $T \in L^{-\alpha/2, 2} (\R^n)$ with support in $\Omega^c$ can be approximated in $L^{-\alpha/2, 2} (\R^n)$ by 
linear combinations of positive measures in $L^{-\alpha/2, 2} (\R^n)$, with support in $\Omega^c$. Then $\langle f, T \rangle =0$ for all such $T$.  Every such $T$ vanishes on $C^{\infty}_0 (\Omega)$ by the support assumption, and hence vanishes on the closure $L^{\alpha/2,2}_0 (\Omega)$.  By the Hahn-Banach theorem, if $f \not\in L^{\alpha/2, 2}_0 (\Omega)$, there would be a $T \in L^{-\alpha/2, 2} (\R^n)$ vanishing on $L^{\alpha/2, 2}_0 (\Omega)$ but not on $f$, which we have seen is impossible.  Hence $f \in L^{\alpha/2, 2}_0 (\Omega)$. 
\end{proof}

For $\mu$ a finite positive measure on $\Omega$, there exists a positive measure $\mu^{\, \prime}$, called the balayage of $\mu$ onto $\Omega^c$, such that $\mu^{\, \prime}$ is supported in $\Omega^c$,
\begin{equation} \label{bal1}
I_{\alpha} \mu (x) \geq I_{\alpha} \mu^{\, \prime} (x) , \,\, \mbox{for all} \,\,\, x \in \R^n, 
\end{equation}
\begin{equation} \label{bal2}
I_{\alpha} \mu (x) =I_{\alpha} \mu^{\, \prime} (x) , \,\, \alpha-q.e.  \,\, \mbox{on} \,\, \Omega^c, 
\end{equation} 
and
\begin{equation}\label{balayagedef}
G \mu (x) = I_{\alpha} \mu (x) - I_{\alpha} \mu^{\, \prime} (x). 
\end{equation}
(See \cite{L}, Sec. IV.6, no. 24-25, and Sec. V.1, no. 2). Also $\mu^{\, \prime} $ is a finite measure with $\mu^{\, \prime}(\R^n) \le \mu (\Omega)$; this fact follows from \cite{L}, equation (4.5.5), and p. 263, lines 11 and 13.  By (\ref{bal2}), $G\mu =0 \alpha$-q.e. on $\Omega^c$.   If $\mu \in L^{- \alpha/2, 2} (\R^n)$ then $\mu^{\, \prime} \in  L^{- \alpha/2, 2} (\R^n)$, since  
\[ \int_{\R^n} I_{\alpha} \mu^{\, \prime} \,d \mu^{\, \prime} \leq \int_{\R^n} I_{\alpha} \mu \,d \mu^{\, \prime} = \int_{\R^n} I_{\alpha} \mu^{\, \prime} \,d \mu \leq \int_{\R^n} I_{\alpha} \mu\,d \mu < \infty,\]
by (\ref{energy})-(\ref{bal1}).  Hence if $\mu \in L^{-\alpha/2, 2} (\R^n)$, then $ I_{\alpha} \mu$ and $I_{\alpha} \mu^{\, \prime}$ are  $\alpha$-quasicontinuous, so $G \mu$ is $\alpha$-quasicontinuous. 

Most of the conclusions of Lemmas \ref{finiteenergy}-\ref{Gbnded} are more or less implicit in \cite{L}, but in somewhat different language from what we require.

\begin{Lemma} \label{finiteenergy}
Suppose $\mu \in L^{-\alpha/2, 2} (\R^n)$ is a finite positive measure on $\Omega$. Then 
$G \mu \in L^{\alpha/2, 2}_0 (\Omega)$, $\int_{\Omega} G \mu \, d \mu < \infty$, and

\begin{equation} \label{Gmuidentity}
 \Vert G \mu \Vert_{L^{\alpha/2, 2}_0 (\Omega)}^2 = \int_{\Omega} G \mu \, d \mu .  
 \end{equation}
\end{Lemma}

\begin{proof}  Since $\mu, \mu^{\, \prime} \in L^{-\alpha/2, 2} (\R^n)$ and  $I_{\alpha}$ maps $L^{-\alpha/2, 2} (\R^n)$ into $L^{\alpha/2, 2} (\R^n)$, we have $G\mu \in L^{\alpha/2, 2} (\R^n)$ by (\ref{balayagedef}).  Then $G \mu \in L^{\alpha/2, 2}_0 (\Omega)$ by (\ref{bal2}), (\ref{balayagedef}),  and Lemma \ref{DenyCorollary}, because $G \mu$ is quasicontinuous as noted above.  Applying (\ref{bal2})-(\ref{balayagedef}), (\ref{energy}), and the fact that $I_{\alpha}$ maps $L^{- \alpha/2, 2} (\R^n)$ isometrically to $L^{ \alpha/2, 2} (\R^n)$, we have 
\[  \int_{\Omega} G \mu \, d \mu = \int_{\R^n} I_{\alpha} (\mu - \mu^{\, \prime}) \, d \mu = \int_{\R^n} I_{\alpha} (\mu - \mu^{\, \prime}) \, d (\mu - \mu^{\, \prime}) \]
\[  = \Vert \mu - \mu^{\, \prime} \Vert^2_{L^{- \alpha/2, 2} (\R^n)} = \Vert I_{\alpha} (\mu - \mu^{\, \prime}) \Vert^2_{L^{ \alpha/2, 2} (\R^n)} = \Vert G \mu \Vert_{L^{\alpha/2, 2} (\Omega)}^2. \]
\end{proof}

The following lemma is an analogue of (\ref{energy}) for $\Omega$.  The proof is more complicated because $G^{(\alpha)}$ does not satisfy a semi-group property.   

\begin{Lemma} \label{charSobNegindex}  Let $\mu$ be a finite positive Borel measure which is compactly supported in $\Omega$.  Let $\alpha \in (0,2)$ and let $G = G^{(\alpha)}$ be the Green's operator of order $\alpha$.  Then the following are equivalent:

(i) $\mu \in L^{- \alpha/2, 2} (\R^n)$; 

(ii) $\mu \in L^{- \alpha/2, 2} (\Omega)$;

(iii) $\int_{\Omega} G \mu \, d \mu < \infty$.

\noindent{If} these conditions hold, then 

\begin{equation} \label{negativeindex}
\int_{\Omega} G \mu \, d \mu = \Vert \mu \Vert_{L^{- \alpha/2, 2}(\Omega)}^2.
\end{equation}
\end{Lemma}

\begin{proof}  Since $L^{\alpha/2, 2}_0 (\Omega) \subseteq L^{\alpha/2, 2}_0 (\R^n)$, (i) implies (ii).  To show that (ii) implies (i), suppose $\mu \in L^{- \alpha/2, 2} (\Omega)$.  Let $\varphi \in C^{\infty}_0 (\R^n)$.  Let $\gamma \in C^{\infty}_0 (\Omega)$ satisfy $\gamma =1$ on the support of $\mu$.  Then by Plancherel's theorem,
\[  \Vert \varphi \gamma \Vert_{L^{\alpha/2, 2} (\R^n)} = \Vert |\xi|^{\alpha/2} \hat{\varphi} \ast \hat{\gamma}  \Vert_{L^2 (\R^n)} \leq c_{\gamma} \Vert |\xi|^{\alpha/2} M ( \hat{\varphi} )(\xi) \Vert_{L^2 (\R^n)} \]
\[ \leq c_{\gamma, \alpha, n} \Vert |\xi|^{\alpha/2}\hat{\varphi} (\xi) \Vert_{L^2 (\R^n)} = c_{\gamma, \alpha, n} \Vert \varphi \Vert_{L^{\alpha/2, 2} (\R^n)}, \] 
where $M$ is the Hardy-Littlewood maximal function (see e.g., \cite{S}, p. 63).  The second inequality holds because the maximal function is bounded on $L^2 (\R^n, |\xi|^{\alpha} \, d\xi)$ because $|\xi|^{\alpha}$ is an $A_2$-weight, since $\alpha <n$ (see \cite{M}).  Since $\varphi \gamma = \varphi$ on the support of $\mu$, and $\varphi \gamma \in C^{\infty}_0 (\Omega)$,
\[  \left| \int_{\R^n} \varphi \, d \mu \right| = \left| \int_{\R^n} \varphi \gamma \, d \mu \right| \leq \Vert \mu \Vert_{L^{- \alpha/2, 2} (\Omega)} \Vert \varphi \gamma \Vert_{L^{\alpha/2, 2}_0 (\R^n)} \]
\[ \leq c_{\gamma, \alpha, n} \Vert \mu \Vert_{L^{- \alpha/2, 2} (\Omega)} \Vert \varphi  \Vert_{L^{\alpha/2, 2}_0 (\R^n)} . \]
Hence $\mu$ extends from the dense subspace $C^{\infty}_0 (\R^n)$ to define a bounded linear functional on $L^{\alpha/2, 2}_0 (\R^n)$.  Thus (i) holds.

If $\mu \in L^{- \alpha/2, 2} (\R^n)$, then by (\ref{balayagedef}) and (\ref{energy}),
\[ \int_{\Omega} G \mu \, d \mu \leq \int_{\R^n} I_{\alpha} \mu \, d \mu = \Vert \mu \Vert_{L^{- \alpha/2, 2} (\R^n)}^2 . \]
Hence (i) implies (iii).

To show that (iii) implies (i), assume that $\int_{\Omega} G \mu \, d \mu < \infty$.  We  show that $\int_{\R^n} I_{\alpha} \mu \, d \mu < \infty$, hence, by (\ref{energy}), $\mu \in L^{- \alpha/2, 2} (\R^n)$.  By (\ref{bal1}) and  (\ref{balayagedef}), 
\[ \int_{\R^n} I_{\alpha} \mu \, d \mu =  \int_{\Omega}  G\mu \, d \mu + \int_{\Omega} I_{\alpha} \mu^{\, \prime} \, d \mu , \]
so it suffices to show that $\int_{\R^n} I_{\alpha} \mu \, d \mu^{\, \prime}  = \int_{\R^n} I_{\alpha} \mu^{\, \prime} \, d \mu  < \infty$. Let $K$ denote the support of $\mu$ and let $O $ be an open set such that $K \subseteq O$ and $\overline{O} \subseteq \Omega$. Select a point $x_0 \in O$ such that $I_{\alpha} \mu (x_0) < \infty$; such a point exists because $I_{\alpha} \mu $ is finite a.e. (\cite{L}, p. 61).  Since the distance between $O$ and $\Omega^c$ is positive, then for $x \in O$ and $y \in \Omega^c$, we have $|x-y| \approx |x_0 - y|$ with constants uniform over $x,y$, and hence there is a constant $C$ such that for all $y \in \Omega^c$,
\[  I_{\alpha} \mu (y) = c_{n, \alpha} \int_K \frac{d \mu (x)}{|x-y|^{n- \alpha} } \leq C \int_K \frac{d \mu (x)}{|x_0-y|^{n- \alpha} } = \frac{C \mu(K)}{|x_0 -y|^{n - \alpha}}. \]
Hence, using (\ref{bal1}) and the fact that $\mu^{\, \prime}$ is supported in $\Omega^c$, 
\[  \int_{\R^n} I_{\alpha} \mu \, d \mu^{\, \prime} \leq C \mu (K) \int_{\R^n} \frac{d \mu^{\, \prime} (y)}{|x_0 -y|^{n - \alpha}} \]
\[ = C_1 \mu (K) I_{\alpha} \mu^{\, \prime} (x_0) \leq C_1 \mu (K) I_{\alpha} \mu (x_0) < \infty .\]

To prove (\ref{negativeindex}) for $\mu \in L^{- \alpha/2, 2} (\R^n)$, as in the proof of Lemma \ref{finiteenergy}, we have 
\[  \int_{\Omega} G \mu \, d \mu  = \Vert \mu - \mu^{\, \prime} \Vert^2_{L^{- \alpha/2, 2} (\R^n)} . \]

Let $\mathcal{A}$ denote the set of linear combinations of positive measures belonging to $L^{-\alpha/2, 2} (\R^n)$ and supported in $\Omega^c$.  We claim that 
\begin{equation} \label{extremal}
\Vert \mu - \mu^{\, \prime} \Vert^2_{L^{- \alpha/2, 2} (\R^n)} = \inf_{\lambda \in \mathcal{A}} \Vert \mu - \lambda \Vert^2_{L^{- \alpha/2, 2} (\R^n)}, 
\end{equation}
i.e., that $\mu^{\, \prime}$ is extremal for the right side of (\ref{extremal}).  A similar statement where the supremum is over positive measures $\lambda$ is proved in \cite{L}, Lemma 4.4, but the extension involving linear combinations is not difficult, as follows.  Let $\langle \, , \, \rangle_{*}$ denote the inner product in $L^{- \alpha/2, 2} (\R^n)$:
\[  \langle \nu_1 ,\nu_2 \rangle_{*} = \int_{\R^n} I_{\alpha/2} \nu_1 I_{\alpha/2} \nu_2 \, dx = \int_{\R^n} I_{\alpha} \nu_1 \, d\nu_2 . \]
Let $\lambda \in \mathcal{A}$. Then 
\[  \Vert \mu - \lambda \Vert^2_{L^{- \alpha/2, 2}(\R^n)} = \Vert \mu - \mu^{\, \prime} \Vert^2_{L^{- \alpha/2, 2}(\R^n)} + 2 \langle  \mu - \mu^{\, \prime}, \mu^{\, \prime} - \lambda  \rangle_{*} + \Vert \mu^{\, \prime} - \lambda \Vert^2_{L^{- \alpha/2, 2}(\R^n)} . \]
However, 
\[ \langle  \mu - \mu^{\, \prime}, \mu^{\, \prime} - \lambda  \rangle_{*} = \int_{\R^n} I_{\alpha} (\mu - \mu^{\, \prime} ) d (\mu^{\, \prime} - \lambda) =0,  \]
by (\ref{bal2}), because $\mu^{\, \prime}, \lambda \in L^{- \alpha/2, 2} (\R^n)$ are supported in $\Omega^c$.  Hence $\Vert \mu - \lambda \Vert^2_{L^{- \alpha/2, 2}(\R^n)}
\geq \Vert \mu - \mu^{\, \prime} \Vert^2_{L^{- \alpha/2, 2}(\R^n)}$, which establishes (\ref{extremal}).  

By Deny's theorem, as noted above, every distribution $\lambda \in L^{-\alpha/2, 2} (\R^n)$ supported in $\Omega^c$ can be approximated in norm by by 
linear combinations of positive measures supported in $\Omega^c$.  Therefore we can replace the class $\mathcal{A}$ with the class $\mathcal{B} = \{ \lambda \in L^{-\alpha/2, 2} (\R^n): \mbox{supp} \, \lambda \subseteq \Omega^c \}$, to obtain
\[  \int_{\Omega} G \mu \, d \mu =  \inf_{\lambda \in \mathcal{B}} \Vert \mu - \lambda \Vert^2_{L^{- \alpha/2, 2} (\R^n)}. \]
By the Hahn-Banach theorem, the last infimum is the square of the norm of $\mu$ in $L^{- \alpha/2, 2} (\Omega)$.  Thus (\ref{negativeindex}) holds.
\end{proof}

We will need a few more facts about $G$ for the next proof.  By \cite{BBK}, p. 15 and 20-21, for $x \in \Omega$, the function $x \rightarrow G(x,y)$ is an $\alpha$-harmonic (hence $C^2$) function of $y$ on $\Omega \setminus \{x\}$, and satisfies a Harnack inequality: for $K$ a compact subset of $\Omega \setminus \{x\}$, there exists a constant $C (K, \Omega, \alpha)>0$ such that $G(x, y_1) \leq C (K, \Omega, \alpha) G(x, y_2)$ for all $y_1, y_2 \in K$.  (That this inequality holds even if $x$ and $y$ are in different connected components of $K$ is a remarkable feature that does not hold in the classical case $\alpha =2$.) The definition of the Green's function shows that $\lim_{y \rightarrow x} G(x,y) = + \infty$.  This fact and the Harnack inequality show that we have the strict inequality $G(x,y) >0$ for each $x,y \in \Omega$.  Then applying Harnack's inequality again shows that for any $x \in \Omega$ and $K \subset \Omega$ compact, $G(x,y)$ is bounded away from $0$ on $K$: more precisely, 
\[ C(x, K)\equiv \inf_{y \in K} G(x,y) >0 .\]
In particular, if $\mu$ is a positive measure on $\Omega$ and $\mu_K$ is the restriction of $\mu$ to $K$, then 
\begin{equation} \label{Harnackconsequence}
G \mu_K (x) = \int_K G(x,y) \, d \mu (y) \geq C(x, K) \mu (K).  
\end{equation}
Also, $C(x,K)$ is a measurable function of $x \in \Omega$, since, for a countable dense subset $\{y_i \}$ of $K$ and $t\in \R$, we have $\{ x \in \Omega: C(x, K) < t \} = \cup_i \{ x \in \Omega: G(x, y_i) < t \},$ by the continuity of $G(x,y)$ for $y \neq x$.   

We can remove the compact support and/or finiteness assumptions on $\mu$ for parts of the last two lemmas by limiting arguments.

\begin{Lemma} \label{Gbnded}
(A) Suppose $\mu$ is a positive measure on $\Omega$. Then the following are equivalent:

\vspace{0.1in}

(i) $\mu \in L^{-\alpha/2, 2} (\Omega)$; 

(ii) $\int_{\Omega} G\mu \, d \mu < \infty$;   

(iii) $G \mu \in  L^{\alpha/2, 2}_0 (\Omega)$.

\vspace{0.1in}

\noindent{If} these conditions hold, then $G \mu$ is $\alpha$-quasicontinuous and 
\begin{equation} \label{normidentity}
  \Vert G \mu \Vert_{ L^{\alpha/2, 2}_0 (\Omega)}^2 = \int_{\Omega} G \mu \, d \mu = \Vert  \mu \Vert_{ L^{-\alpha/2, 2} (\Omega)}^2 . 
\end{equation}
Also, if $\mu_1, \mu_2$ are positive measures belonging to $L^{-\alpha/2, 2} (\Omega)$, then (\ref{normidentity}) holds for $\mu= \mu_1 - \mu_2$ as well.

\vspace{0.1in}

(B) We can extend $G$ by continuity to an isometry from $L^{-\alpha/2, 2} (\Omega)$ onto $L^{\alpha/2, 2}_0 (\Omega)$.  Then for all $\mu \in L^{-\alpha/2, 2} (\Omega)$,
\begin{equation}\label{normidentitygen}
 \Vert G \mu \Vert_{ L^{\alpha/2, 2}_0 (\Omega)}^2 = \langle  \mu , G \mu \rangle = \Vert  \mu \Vert_{ L^{-\alpha/2, 2} (\Omega)}^2 , 
\end{equation}
where $\langle , \rangle$ donotes the duality pairing between $L^{-\alpha/2, 2} (\Omega)$ and $L^{\alpha/2, 2}_0 (\Omega)$.
\end{Lemma}

\begin{proof} (A) First suppose $ \mu \in L^{-\alpha/2, 2} (\Omega)$ is a positive measure.  Let $K$ be any compact subset of $\Omega$ and let $\mu_K$ be the restriction of $\mu$ to $K$.  Then we can find $\varphi \in C^{\infty}_0 (\Omega)$ such that $\varphi \geq 0$ and $\varphi =1$ on $K$.  Then 
\[ \mu (K) \leq \int_{\Omega} \varphi \, d \mu \leq \Vert \mu \Vert_{L^{-\alpha/2, 2} (\Omega)} \Vert \varphi \Vert_{L^{\alpha/2, 2}_0 (\Omega)} < \infty.\]
Hence $\mu$ is finite on compact subsets of $\Omega$.

We claim that $\mu_K \in  L^{-\alpha/2, 2} (\Omega)$.  Since $\mu_K$ is finite and compactly supported in $\Omega$, it suffices to show that $\mu_K \in  L^{-\alpha/2, 2} (\R^n)$, by Lemma \ref{charSobNegindex}.  Using (\ref{energy}), 
\[ \Vert \mu_K \Vert_{L^{-\alpha/2, 2} (\R^n)} = \Vert I_{\alpha/2} \mu_K \Vert_{L^2 (\R^n)} = \sup_{\varphi \in \mathcal{A}} \int_{\R^n} (I_{\alpha/2} \mu_K ) \varphi \, dx, \]
where $\mathcal{A} = \{ \varphi \in C^{\infty}_0 (\R^n) : \varphi \geq 0 \, \, \mbox{and} \,\, \Vert \varphi \Vert_{L^2} (\R^n) \leq 1 \}$.  Since $I_{\alpha/2}$ has a symmetric kernel, 
\[ \Vert \mu_K \Vert_{L^{-\alpha/2, 2} (\R^n)} = \sup_{\varphi \in \mathcal{A} } \int_{\R^n} I_{\alpha/2} \varphi \, d \mu_K .   \]
Let $\gamma \in C^{\infty}_0 (\Omega)$ satisfy $\gamma \geq 0$ and $\gamma =1$ on $K$.  Then
\[ \Vert \mu_K \Vert_{L^{-\alpha/2, 2} (\R^n)} = \sup_{\varphi \in \mathcal{A} } \int_{\R^n} \gamma I_{\alpha/2} \varphi \, d \mu_K  \leq \sup_{\varphi \in \mathcal{A} } \int_{\R^n} \gamma I_{\alpha/2} \varphi \, d \mu,  \]
since the integrand is non-negative on $\Omega$.  By the same argument as in the proof of Lemma \ref{charSobNegindex}, we have 
\[ \Vert \gamma I_{\alpha/2} \varphi\Vert_{L^{\alpha/2, 2} (\R^n)} \leq c_{\gamma, \alpha, n}  \Vert  I_{\alpha/2} \varphi\Vert_{L^{\alpha/2, 2} (\R^n)} = c_{\gamma, \alpha, n} \Vert \varphi \Vert_{L^2 (\R^n)} \leq c_{\gamma, \alpha, n} . \]
By Lemma \ref{DenyCorollary}, then, $ \gamma I_{\alpha/2} \varphi \in L^{\alpha/2, 2}_0 (\Omega)$.  Hence 
\[ \Vert \mu_K \Vert_{L^{-\alpha/2, 2} (\R^n)} \leq \sup_{\varphi \in \mathcal{A}} \Vert \mu \Vert_{L^{-\alpha/2, 2}(\Omega)} \Vert \gamma I_{\alpha/2} \varphi\Vert_{L^{\alpha/2, 2} (\R^n)} \]
\[ \leq c_{\gamma, \alpha, n} \Vert \mu \Vert_{L^{-\alpha/2, 2}(\Omega)} < \infty. \]

Since $\mu_K \in L^{-\alpha/2, 2}(\Omega)$, Lemma \ref{finiteenergy} gives that $G \mu_K \in L^{\alpha/2, 2}_0 (\Omega)$, and $\Vert G \mu_K \Vert_{L^{\alpha/2, 2} (\R^n)}^2 = \int_{\Omega} G \mu_K \, d \mu_K < \infty$. Hence
\[  \int_{\Omega} G \mu_K \, d \mu_K \leq \int_{\Omega} G \mu_K \, d \mu \leq \Vert \mu \Vert_{L^{-\alpha/2, 2} (\Omega)} \Vert G \mu_K \Vert_{L^{\alpha/2, 2}_0 (\R^n)}\]
\[  = \Vert \mu \Vert_{L^{-\alpha/2, 2} (\Omega)}  \left(  \int_{\Omega} G \mu_K \, d \mu_K \right)^{1/2} .\]
Dividing and squaring gives 
\[   \int_{\Omega} G \mu_K \, d \mu_K \leq \Vert \mu \Vert^2_{L^{-\alpha/2, 2}_0 (\Omega)}  . \]

Let $\{ K_j \}_{j=1}^{\infty}$ be an increasing sequence of compact subsets of $\Omega$ with $\cup_{j=1}^{\infty} K_j = \Omega$.  Let $\mu_j$ be the restriction of $\mu$ to $K_j$. Then by our last conclusion, we have $\int_{\Omega} G \mu_j \, d \mu_j \leq \Vert \mu \Vert^2_{L^{-\alpha/2, 2}_0 (\Omega)}$.  
The monotone convergence theorem shows that $\int_{\Omega} G \mu \, d \mu \leq \Vert \mu \Vert^2_{L^{-\alpha/2, 2} (\Omega)}$. Hence (i) implies (ii).

Now suppose $\mu$ is a positive measure on $\Omega$ and $\int_{\Omega} G \mu \, d \mu < \infty$. We first note that $\mu$ is finite on any compact subset $K$ of $\Omega$. Assume $\mu(K) >0$ and let $\mu_K$ be the restriction of $\mu$ to $K$.  Let $C(x,K)= \inf_{y \in K} G(x,y)$.  We noted earlier that $C(x,K)>0$ for all $x \in \Omega$.  Hence $C(K) = \int_K C(x,K) \, d \mu (x) >0$.  By (\ref{Harnackconsequence}), 
\[    C(K) \mu(K) = \int_K C(x,K) \mu (K) \, d \mu (x) \leq \int G \mu_K \, d \mu_K \leq \int G\mu \, d \mu < \infty . \]
Hence $\mu(K) < \infty$. 

Let $K_j$ and $\mu_j$ be as above.  Since $\int_{\Omega} G \mu_j \, d \mu_j < \infty$, Lemma \ref{charSobNegindex} implies that $\mu_j \in L^{-\alpha/2, 2} (\Omega)$ and 
\begin{equation} \label{muj}
  \Vert \mu_j \Vert^2_{L^{- \alpha/2, 2} (\Omega)} = \int_{\Omega}  G \mu_j \, d \mu_j . 
\end{equation}
We claim that $\{  \mu_j \}_{j=1}^{\infty}$ is a Cauchy sequence in $L^{-\alpha/2, 2} (\Omega)$.  Suppose $\ell >j$.  Then by Lemma \ref{charSobNegindex} applied to the finite positive measure $\mu_{\ell} - \mu_j$ (which is $\mu$ restricted to $K_{\ell} \setminus K_j$), 
\[ \Vert  \mu_{\ell} - \mu_j \Vert_{L^{-\alpha/2, 2}(\Omega)}^2 = \int_{\Omega} G (\mu_{\ell} - \mu_j  ) \, d (\mu_{\ell} - \mu_j)  \]
\[  = \int_{\Omega} G ( \mu \chi_{K_{\ell} \setminus K_j})\chi_{K_{\ell} \setminus K_j} \, d \mu \leq  \int_{\Omega} G ( \mu \chi_{\Omega \setminus K_j})\chi_{\Omega \setminus K_j} \, d \mu  \rightarrow 0 \]
as $j \rightarrow \infty$, by the dominated convergence theorem, with dominating function $G \mu$.  
Hence $\mu_j$ converges to some $\mu_0 \in  L^{- \alpha/2, 2} (\Omega)$.  Let $\varphi \in C^{\infty}_0 (\Omega)$.  For $j$ large enough that supp $\varphi \subseteq K_j$, we have $\langle \mu_j, \varphi \rangle = \langle \mu, \varphi \rangle$, so $\langle \mu, \varphi \rangle = \lim_{j \rightarrow \infty} \langle \mu_j, \varphi \rangle = \langle \mu_0, \varphi \rangle$.  Hence $\mu = \mu_0 $ in $\mathcal{D}^{\, \prime} (\Omega)$.  Therefore $\mu \in L^{-\alpha/2, 2} (\Omega)$. Thus (ii) implies (i).

Hence (i) and (ii) are equivalent, and if either holds, letting $j \rightarrow \infty$ in (\ref{muj}) shows that $\int_{\Omega} G \mu \, d \mu = \Vert  \mu \Vert_{ L^{-\alpha/2, 2} (\Omega)}^2$ .  

Now suppose (i) and (ii) hold.  Note that $\mu_j \in L^{-\alpha/2,2} (\R^n)$ since $\int_{\Omega} G \mu_j \, d \mu_j \leq \int_{\Omega} G \mu \, d \mu < \infty$.  
By the same argument as above, only using Lemma \ref{finiteenergy} instead of Lemma \ref{charSobNegindex}, $\{ G \mu_j \}_{j=1}^{\infty}$ is a Cauchy sequence in $L^{\alpha/2, 2}_0 (\Omega)$ and hence $G \mu_j $ converges in $ L^{\alpha/2, 2}_0 (\Omega)$ to some $h$.  Since $ L^{\alpha/2, 2}_0 (\R^n)$ imbeds continuously in $L^{p^*} (\R^n)$, where $p^* = 2n/ (n - \alpha)$ (see e.g., \cite{DPV}, Theorem 6.5), $G \mu_j $ converges to $h$ in $ L^{p^*} (\R^n)$, and hence a subsequence converges almost everywhere to $h$. But $G \mu_j (x)$ increases to $G \mu (x)$ at every $x$, hence $h = G \mu$.  Therefore $G \mu \in L^{\alpha/2, 2}_0 (\Omega)$.  By Lemma \ref{finiteenergy}, \[  \Vert G \mu_j \Vert^2_{L^{\alpha/2, 2}_0 (\Omega)} =  \int_{\Omega} G \mu_j \, d \mu_j, \]
and taking the limit as $j \rightarrow \infty$ and applying the monotone convergence theorem on the right side gives the identity on the left of (\ref{normidentity}).  Hence (i) and (ii) imply (iii).

We now show that (i) and (ii) imply that $G \mu$ is $\alpha$-quasicontinuous.  Indeed, each $G\mu_j$ is quasicontinuous, and $G \mu_j$ converges in $ L^{\alpha/2, 2}$ norm and pointwise to $G \mu$, so the quasicontinuity of $G \mu$ follows from our earlier remarks.

Let us now assume that $G \mu \in L^{\alpha/2, 2}_0 (\Omega)$. Then, as we noted above, $G \mu \in L^{p^*}(\Omega)$, where $p^{*}=2n/(n-\alpha)$. For any $h \in C^\infty_0(\Omega)$, we have (see \cite{BBK}, p. 14) 
\[
G \left((-\Delta)^{\alpha/2} h\right) (x) = h(x), \quad x \in \R^n.
\]
By Fourier inversion and (\ref{pointwisefracLap}), we see easily that
\[
 |(-\Delta)^{\alpha/2} h(x)|\le c (1+|x|)^{-n-\alpha}, \quad x \in \R^n. 
 \]
Therefore $(-\Delta)^{\alpha/2} h\in L^r(\R^n)$ for $r>n/(n+\alpha)$. In particular, 
$(-\Delta)^{\alpha/2} h\in L^{q^*}(\R^n)$, where $q^* =2n/(n+\alpha)$ is the conjugate index to $p^*$.
It follows that 
\[
\iint_{\R^n\times\Omega} G(x, y) \left\vert(-\Delta)^{\alpha/2} h (x)\right\vert d\mu(y) dx =\int_{\R^n} \left\vert(-\Delta)^{\alpha/2} h(x)\right\vert \, G \mu(x) \, dx <\infty.
   \]
Hence, by Fubini's theorem 
\[
\left \vert \int_\Omega h \, d \mu \right\vert =\left \vert  \int_\Omega G \left((-\Delta)^{\alpha/2} h\right) \, d \mu \right\vert =
\left \vert \int_{\R^n} (-\Delta)^{\alpha/2} h \, G \mu \, dx \right\vert 
\]
\[
= \left \vert \langle G \mu, (-\Delta)^{\alpha/2} h\rangle \right\vert \le 
||G \mu||_{L^{\alpha/2, 2}_0 (\R^n)} ||(-\Delta)^{\alpha/2}h||_{L^{-\alpha/2, 2} (\R^n)}\]
\[
=||G \mu||_{L^{\alpha/2, 2}_0 (\Omega)} ||h||_{L^{\alpha/2, 2}_0 (\Omega)}.
\]
Thus, $\mu \in L^{- \alpha/2, 2} (\Omega)$, and so (iii) implies (i).

Now suppose $\mu_1, \mu_2$ are positive measures belonging to $L^{-\alpha/2, 2} (\Omega)$.  By (\ref{normidentity}), 
\[   \Vert G (\mu_1 + \mu_2) \Vert_{ L^{\alpha/2, 2}_0 (\Omega)}^2 = \int_{\Omega} G (\mu_1 + \mu_2) \, d (\mu_1 + \mu_2)  .\]
Expanding and using (\ref{normidentity}) to cancel the non-diagonal terms, we obtain
\begin{equation} \label{polarization}
 \langle G \mu_1, G \mu_2 \rangle_{\alpha/2} +  \langle G \mu_2, G \mu_1 \rangle_{\alpha/2} =  \int_{\Omega} G \mu_1  \, d  \mu_2 + \int_{\Omega} G \mu_2  \, d  \mu_1 , 
\end{equation}
where $\langle , \rangle_{\alpha/2}$ denotes the inner product in $L^{\alpha/2, 2} (\Omega)$.  To obtain the first identity in (\ref{normidentity}) for $\mu = \mu_1 - \mu_2$, expand both sides as above, only with $\mu_1 - \mu_2$ in place of $\mu_1 + \mu_2$, and use (\ref{polarization}), and (\ref{normidentity}) for $\mu_1$ and $\mu_2$.  The second identity in (\ref{normidentity}) is proved in the same way.

\vspace{0.1in}

(B)  We now have that $G$ is an isometry from the linear combinations of positive measures in $L^{- \alpha/2, 2} (\Omega)$ to $L^{\alpha/2, 2}_0 (\Omega)$.  By Deny's theorem (\cite{D}, Theorem II.2), the linear combinations of positive measures in $L^{- \alpha/2, 2} (\Omega)$ are dense.  Hence we can extent $G$ to be an isometry from all of $L^{- \alpha/2, 2} (\Omega)$ to $L^{\alpha/2, 2}_0 (\Omega)$.  To prove (\ref{normidentitygen}), suppose $\mu \in L^{- \alpha/2, 2} (\Omega)$.  Then there exists a sequence $\mu_j$ of linear combinations of positive measures converging in $L^{- \alpha/2, 2} (\Omega)$ norm to $\mu$.  Then $G \mu_j$ converges to $G \mu$ in $L^{ \alpha/2, 2}_0 (\Omega)$, hence 
\[  \langle \mu, G \mu  \rangle = \lim_{j \rightarrow \infty} \langle \mu_j, G \mu_j  \rangle =   \lim_{j \rightarrow \infty} \int  G \mu_j \, d \mu_j \]
\[ = \lim_{j \rightarrow \infty} \Vert \mu_j \Vert^2_{L^{- \alpha/2, 2} (\Omega)} = \Vert \mu \Vert^2_{L^{- \alpha/2, 2} (\Omega)} . \]
The other identity in (\ref{normidentitygen}) follows now because $G$ is an isometry.  

To prove that $G$ maps $L^{- \alpha/2, 2} (\Omega)$ onto $L^{\alpha/2, 2}_0 (\Omega)$, suppose otherwise.  Then by the Hahn-Banach theorem, there exists $\nu \in L^{- \alpha/2, 2} (\Omega)$ which is not identically zero on $L^{\alpha/2, 2}_0 (\Omega)$, but vanishes on the image of $L^{- \alpha/2, 2} (\Omega)$, hence on $G\nu$.  Therefore 
\[ \Vert \nu \Vert_{L^{- \alpha/2, 2} (\Omega)}^2 = \langle \nu, G \nu \rangle =0, \]
so $\nu =0$ in $L^{- \alpha/2, 2} (\Omega)$, a contradiction.  \end{proof}

The last result defines $G \mu$, for $\mu \in L^{- \alpha/2, 2} (\Omega)$, as an element of the Sobolev space $L^{\alpha/2, 2}_0 (\Omega)$, hence a.e.  When considering $G \mu$ pointwise, we now define $G\mu$ $\alpha$-q.e. by choosing an $\alpha$-quasicontinuous representative $\widetilde{G \mu}$ of the equivalence class of $G \mu$ in $L^{\alpha/2, 2}_0 (\Omega)$, and defining $G \mu$ to be $\widetilde{G \mu}$.  Any other $\alpha$-quasicontinuous representative will agree $\alpha$-q.e., by \cite{AH}, Ch. 6.1, so $G \mu$ is now defined as an equivalence class under the equivalence relation of equality q.e.  This convention will allow us to avoid replacing $G \mu$ with $\widetilde{G \mu}$ at several points later, and, more importantly, will allow us to interpret the identity $u = G(u \omega + \nu)$ as holding pointwise $\alpha$-q.e. rather than just a.e. If $\mu$ is a postive measure, $G \mu$ is defined for all $x$ as $\int_{\Omega} G(x,y) \, d \mu (y)$, which is finite $\alpha$-q.e. and is $\alpha$-quasicontinuous, by Lemma (\ref{Gbnded}). If $\mu \in L^{- \alpha/2, 2} (\Omega)$ is a linear combination of positive measures, $G \mu$ is defined where all of the measures in the linear combination are finite, hence $\alpha$-q.e., and $G \mu$ is $\alpha$-quasicontinuous.  Hence our pointwise definition of $G \mu$ for general $\mu \in L^{- \alpha/2, 2} (\Omega)$ is consistent with pointwise definitions considered previously.  
 
\begin{Lemma}\label{Tbetaequiv} Let $T$ be the operator in (\ref{defT}). Let $\beta >0$.  The following are equivalent:

(i) $T$ maps $L^2 (\omega)$ to itself boundedly with $||T|| \, \leq \, \beta^2$;

(ii) (\ref{embedd}) holds, i.e., 
\[ \Vert h \Vert_{L^2 (\omega)} \leq \beta \Vert  h \Vert_{L^{\alpha/2, 2}_0 (\Omega)}, \,\,\, \mbox{for all} \,\,\, h \in C^{\infty}_0 (\Omega);  \]

(iii) \begin{equation}\label{embedd2}
\Vert \tilde{u} \Vert_{L^2 (\omega)} \leq \beta \Vert u \Vert_{L^{\alpha/2,2}_0 (\Omega)}, \,\,\, \mbox{for all} \,\,\, u \in L^{\alpha/2 ,2}_0 (\Omega),
\end{equation}
where $\tilde{u}$ denotes any quasicontinuous representative of $u$ in $L^{\alpha/2,2}_0 (\Omega)$;

(iv) \begin{equation} \label{dualform}
\Vert h \omega \Vert_{L^{- \alpha/2, 2} (\Omega)} \leq \beta \Vert h \Vert_{L^2 (\omega)},  \,\,\, \mbox{for all} \,\,\, h \in L^2 (\omega).
\end{equation}

\end{Lemma} 

\begin{proof}  First we show the equivalence of (ii) and (iii).  Suppose (ii) holds.  Let  $u \in L^{\alpha/2,2}_0 (\Omega)$.  There exists a sequence $\{h_n \}_{n=1}^{\infty}$ with $h_n \in C^{\infty}_0 (\Omega)$ such that $||h_n - u ||_{L^{\alpha/2,2}_0 (\R^n)} \to 0$. Then, as noted earlier, there exists a subsequence $h_{n_k} \to \tilde u$  q.e. in $\R^n$,  where $\tilde u$ is a quasicontinuous representative of $u$.  By (ii), $h_{n_k} $ is a Cauchy sequence  in $L^2(\omega)$, hence $h_{n_k} \to u_0$ in $L^2(\omega)$ for some $u_0 \in L^2(\omega)$. Hence, replacing  $h_{n_k} $ with a further subsequence, we see that $h_{n_k} \to u_0$, $d \omega$ a.e., 
and at the same time $h_{n_k} \to \tilde u$ q.e. 

Let $K \subset \Omega$ be compact.  If $\varphi \in C^{\infty}_0 (\Omega)$ and $\varphi \geq 1$ on $K$, then by (\ref{embedd}), 
\[  \omega (K) \leq \Vert \varphi \Vert_{L^2 (\omega)}^2 \leq \beta^2 \Vert \varphi \Vert_{L^{\alpha/2, 2}_0 (\Omega)}^2 , \]
hence, taking the infimum over such $\varphi$,
\begin{equation} \label{abscont}
 \omega(K) \le \beta^2 \,  \text{cap}_{\alpha} (K, \Omega) .
\end{equation} 
For compact sets $K \subset \Omega$, we have 
\[ \text{cap}_{\alpha} (K, \Omega) \asymp \text{cap}_{\alpha} (K, \R^n), \]
where the constants of equivalence depend on $\text{dist}(K, \Omega^c)$, by the same argument as at the beginning of the proof of Lemma \ref{charSobNegindex}. In particular, $\omega$ is absolutely continuous with respect to $\text{cap}_{\alpha} (\cdot, \R^n)$. 

It follows that $h_{n_k} \to \tilde u$ $d \omega$-a.e. on $K$, and consequently 
$u_0=\tilde u$ $d \omega$-a.e. on $K$, for any compact set $K \subset \Omega$. 
Thus, using  (\ref{embedd}) with $h=h_{n_k}$, and 
 letting $n_k \to \infty$, we arrive at 
 \begin{equation}\label{embedd-b}
\Vert \tilde u \Vert_{L^2 (\omega_K)} \leq \beta \Vert u \Vert_{L^{\alpha/2,2}_0 (\Omega)}.
\end{equation}
Since $K$ is an arbitrary compact subset of $\Omega$, we deduce that 
 (\ref{embedd2}) holds. 
Hence (ii) implies (iii).  The converse is trivial.

Next we show that (ii) and (iv) are equivalent.  Suppose (ii) holds, $h \in L^2 (\omega)$, and, to begin with, that $h \geq 0$.  Let
\[ \mathcal{A} = \{ \varphi \in C^{\infty}_0 (\Omega) : \Vert \varphi \Vert_{L^{\alpha/2, 2}_0 (\Omega)} \leq 1 \}  \]
and let $\langle \cdot , \cdot \rangle$ denote the pairing between $L^{\alpha/2, 2}_0 (\Omega)$ and its dual $L^{-\alpha/2, 2} (\Omega)$.  Then   
\[ \Vert h \omega \Vert_{L^{- \alpha/2, 2}(\Omega)} =  \sup_{\varphi \in \mathcal{A}} |\langle h \omega, \varphi \rangle| = \sup_{\varphi \in \mathcal{A}} \left| \int_{\Omega} \varphi h \, d \omega  \right| . \]
For $\varphi \in \mathcal{A}$, we have $\Vert \varphi \Vert_{L^2 (\omega)}\leq \beta$, by (ii).  Hence
\[ \Vert h \omega \Vert_{L^{- \alpha/2, 2}(\Omega)} \leq   \sup_{g: \Vert g \Vert_{L^2 (\omega)} \leq \beta } \left|\int_{\Omega} gh \, d \omega \right| = \beta \Vert h \Vert_{L^2 (\omega)} . \]
The same argument holds if $h \leq 0$ (since we still have $|h\omega | = -h \omega \in L^{- \alpha/2, 2}(\Omega)$, which is needed to justify the identity $\langle h \omega, \varphi \rangle = \int_{\Omega} \varphi h \, d \omega$).  For a general $h \in L^2 (\omega)$, we have $h^+, h^- \in L^2 (\omega)$, so by what we have just shown, $h^+ \omega, h^- \omega \in L^{- \alpha/2, 2}(\Omega)$, hence $|h \omega|  \in L^{- \alpha/2, 2}(\Omega)$.  Then the same argument yields (\ref{dualform}) in the general case.  

Now suppose (iv) holds and $h \in C^{\infty}_0 (\Omega)$.  If $f \in L^2 (\omega)$, then $|f \omega| \in L^{-\alpha/2, 2}(\Omega)$ by (iv), so 
\[  \Vert h \Vert_{L^2 (\omega)} = \sup_{\Vert f \Vert_{L^2 (\omega)} \leq 1} \left| \int_{\Omega} h f \, d \omega \right| = \sup_{\Vert f \Vert_{L^2 (\omega)} \leq 1} \left|
\langle f \omega, h \rangle \right| \]
\[ \leq \sup_{\Vert f \Vert_{L^2 (\omega)} \leq 1} \Vert f \omega \Vert_{L^{-\alpha/2,2} (\Omega)} \Vert h \Vert_{L^{\alpha/2, 2}_0 (\Omega)} \leq \beta  \Vert h \Vert_{L^{\alpha/2, 2}_0 (\Omega)},\]
by (iv).  Thus (ii) holds.

We now prove the equivalence of (i) and (iv).  For either direction, we observe that since $T$ is self-adjoint, 
\[ \Vert T \Vert_{L^2 (\omega) \rightarrow L^2 (\omega) }= \sup_{g: \Vert g \Vert_{L^2 (\omega)} \leq 1} |\langle Tg, g \rangle_{\omega} | \]
\[ = \sup_{g: \Vert g \Vert_{L^2 (\omega)} \leq 1} \left|\int_{\Omega} \int_{\Omega} G(x,y) g(y) \, d \omega (y) g(x) \, d \omega (x) \right| .   \]
For $g \in L^2 (\omega)$, let $g^+ = g \chi_{\{ x \in \Omega: g(x)>0\}}$ and $g^- = - g \chi_{x \in \Omega: g(x) <0}$.  

Suppose (i) holds. Then 
\[ \int_{\Omega} G (g^+ \omega) \, g^+ d \omega = \int_{\Omega} \int_{\Omega} G(x,y) g^+(y) \, d \omega (y) g^+(x) \, d \omega (x)  \]
\[ \leq \Vert T \Vert \Vert g^+ \Vert_{L^2 (\omega)}^2 < \infty ,  \]
and similarly for $g^- \omega$.  By Lemma \ref{Gbnded}, $g^+ \omega, g^- \omega \in L^{- \alpha/2, 2} (\R^n)$, so $g \omega \in L^{- \alpha/2, 2} (\R^n)$, and 
\begin{equation} \label{normTeqn}
\Vert T \Vert = \sup_{\Vert g \Vert_{L^2 (\omega)} \leq 1} \int_{\Omega} G(g  \omega) g \, d\omega = \sup_{\Vert g \Vert_{L^2 (\omega)} \leq 1} \Vert g  \omega \Vert_{L^{- \alpha/2, 2} (\Omega)}^2 .
\end{equation}
Hence
\begin{equation} \label{negindexTest}
 \Vert g \omega \Vert_{L^{- \alpha/2, 2} (\Omega)}^2 \leq \Vert T \Vert   \Vert g \Vert_{L^2 (\omega)}^2 \leq \beta^2 \Vert g \Vert_{L^2 (\omega)}^2 , 
\end{equation}
for all $g \in L^2 (\omega)$. 

Conversely, if (iv) is assumed, and $g \in L^2 (\omega)$, then $g^+ \omega, g^- \omega \in L^{- \alpha/2, 2} (\Omega)$.  Then (\ref{normTeqn}) holds by Lemma \ref{Gbnded}. Then, using (\ref{dualform}) again, $\Vert T \Vert \leq \beta^2 $.

\end{proof}

{\bf Remark.}  We note that the equivalence of (ii) and (iii) holds in the range $0<\alpha\le 2$, $\alpha<n$.  If $\alpha=2$ and $n=2$, then we have to assume that the domain $\Omega$ is a Green domain (with non-trivial Green's function), since otherwise 
the inequality fails for constant $h \in L^{1,2}_0 (\Omega)$, unless $\omega=0$. Then the same proof 
works if the Riesz capacity $\text{cap}_{\alpha} (\cdot, \R^2)$  is replaced 
with the Bessel capacity 
$\text{Cap}_{\alpha} (\cdot, \R^2)$.  This case was considered for $\sigma \in L^1_{{\rm loc}} (\omega)$, without providing details,  in 
\cite{FV} 

The same observation applies to the case $\alpha=n$ for $n>2$, where Bessel capacities $\text{Cap}_{\alpha} (\cdot, \R^n)$ can be used if the domain $\Omega$ is $n$-Green. However, we consider here only the case $\alpha<n$.

\vspace{0.1in}

For $\nu \in L^{-\alpha/2,2} (\Omega)$, we say that $u$ is a weak solution of equation (\ref{u0equation}) if $u \in L^{\alpha/2, 2}_0 (\Omega)$ and 
\begin{equation} \label{weakform}
\int_{\R^n} (-\triangle)^{\alpha/4} u  (-\triangle)^{\alpha/4} \varphi \, dx = \int_{\R^n} \tilde{u} \tilde{\varphi} \, d \omega + \langle \nu, \varphi \rangle,
\end{equation} 
for all $\varphi \in L^{\alpha/2, 2}_0 (\Omega)$, where $\tilde{u}$ and $\tilde{\varphi}$  are quasicontinuous representatives of $u$ and $\varphi$, respectively, in $L^{\alpha/2, 2}_0 (\Omega)$ and $\langle \cdot , \cdot \rangle$ is the pairing between $L^{-\alpha/2,2} (\Omega)$ and $L^{\alpha/2, 2}_0 (\Omega)$.  A standard application of the Lax-Milgram Theorem shows that if $\Vert T \Vert <1$, then for each $\nu \in L^{-\alpha/2,2} (\Omega)$, there exists a unique weak solution $u$ of (\ref{u0equation}), as follows.  Following the formulation of the Lax-Milgram Theorem in \cite{E}, \S 6.2.1, let $H= L^{\alpha/2, 2}_0 (\Omega)$ and define the bilinear form
\begin{equation} \label{defB}
 B(u,v) = \int_{\R^n} (-\triangle)^{\alpha/4} u  (-\triangle)^{\alpha/4} v \, dx - \int_{\R^n} \tilde{u} \tilde{v} \, d \omega 
\end{equation} 
on $H \times H$.  Let $\Vert T \Vert^{1/2} = \beta <1$.  Then (\ref{embedd2}) holds, which shows that $\int_{\R^n} \tilde{u} \tilde{\varphi} \, d \omega$ is well-defined, i.e., if $\widetilde{u_1}$ and $\widetilde{u_2}$ are quasicontinuous representatives of the same equivalence class in $L^{\alpha/2, 2}_0 (\Omega)$, then $\widetilde{u_1} = \widetilde{u_2} \,\, \omega$-a.e. By the Cauchy-Schwarz inequality,
\[   |B(u,v)| \leq  \Vert u \Vert_{L^{\alpha/2,2}_0 (\Omega)}   \Vert v \Vert_{L^{\alpha/2,2}_0 (\Omega)} + \Vert  \tilde{u} \Vert_{L^2 (\omega)} \Vert  \tilde{v} \Vert_{L^2 (\omega)} \]
\[\leq (1 + \beta^2) \Vert u \Vert_{L^{\alpha/2,2}_0 (\Omega)}   \Vert v \Vert_{L^{\alpha/2,2}_0 (\Omega)}.  \] 
Equation (\ref{embedd2}) also implies the coercivity of $B$:
\[ B(u,u) = \Vert u \Vert^2_{L^{\alpha/2,2}_0 (\Omega)} - \Vert \tilde{u} \Vert^2_{L^2 (\omega)} \geq (1- \beta^2)\Vert u \Vert^2_{L^{\alpha/2,2}_0 (\Omega)}. \]  
Then the Lax-Milgram Theorem gives, for each $\nu \in L^{-\alpha/2,2} (\Omega)$, the existence of a unique $u \in L^{\alpha/2, 2}_0 (\Omega)$ such that $B(u, \varphi) = \langle \nu, \varphi \rangle$ for all $\varphi \in L^{\alpha/2, 2}_0 (\Omega)$, which is (\ref{weakform}).

\vspace{0.1in}

\begin{proofof}{Proposition \ref{u0thm}}. (A) \,\, First suppose  $\nu \in L^{-\alpha/2, 2} (\Omega)$ is a positive measure.  Then $G\nu \in L^{\alpha/2, 2}_0 (\Omega) $ and $G \nu$ is $\alpha$-quasicontinuous, by Lemma \ref{Gbnded}.  By (\ref{embedd2}) with $\beta = \Vert T \Vert$ and Lemma \ref{Tbetaequiv}, it follows that $G \nu \in L^2 (\omega)$.  Since $\Vert T \Vert <1$, we have $u_0 \equiv \mathcal{G} \nu  = (I-T)^{-1} G \nu \in L^2 (\omega)$.  Then by (\ref{dualform}) with $\beta = \Vert T \Vert$, we have $u_0 \omega \in L^{-\alpha/2, 2} (\Omega)$.  By Lemma \ref{Gbnded} again, we have $G(u_0 \omega) \in L^{\alpha/2, 2}_0 (\Omega)$ and $G(u_0 \omega)$ is quasicontinuous.  Recall that $u_0 = G (u_0 \omega) + G \nu $ holds pointwise at all points, if we allow infinite values.  Since $G(u_0 \omega)$ and $G \nu$ belong to $L^{\alpha/2, 2}_0 (\Omega)$ and are quasicontinuous, we obtain that $u_0 \in L^{\alpha/2, 2}_0 (\Omega)$ and $u_0$ is quasicontinuous.  Recall that by  Sobolev imbedding (as noted in the proof of Lemma \ref{Gbnded}), it follows that $u_0 \in L^{p^*} (\Omega)$, where $p^{*} = 2n/ (n- \alpha)$.  Hence $u_0 < \infty$ a.e., so $u_0$ is a pointwise solution of $u_0 = G (u_0 \omega) + G \nu $.  Now $\Vert G \nu \Vert_{L^{\alpha/2, 2}_0 (\Omega)} = \Vert \nu \Vert_{L^{-\alpha/2, 2} (\Omega)}$ by Lemma \ref{Gbnded}, and, following the estimates in the above results, 
\[ \Vert G (u_0 \omega) \Vert_{L^{\alpha/2, 2}_0 (\Omega)} =  \Vert u_0 \omega \Vert_{L^{-\alpha/2, 2} (\Omega)}  \leq \Vert T \Vert^{1/2} \Vert u_0 \Vert_{L^2 (\omega)} \]
\[ = \Vert T \Vert^{1/2} \Vert (I-T)^{-1} G \nu \Vert_{L^2 (\omega)} \leq \frac{\Vert T \Vert^{1/2} }{1- \Vert T \Vert} \Vert  G \nu \Vert_{L^2 (\omega)} \]
\[ \leq  \frac{\Vert T \Vert }{1- \Vert T \Vert} \Vert G \nu \Vert_{L^{\alpha/2, 2}_0 (\Omega)} =  \frac{\Vert T \Vert }{1- \Vert T \Vert} \Vert  \nu \Vert_{L^{-\alpha/2, 2} (\Omega)}  . \]
Hence
\[  \Vert u_0 \Vert_{L^{\alpha/2, 2}_0 (\Omega)} =  \Vert G(u_0 \omega) + G \nu \Vert_{L^{\alpha/2, 2}_0 (\Omega)}\]
\[  \leq \left(    \frac{\Vert T \Vert }{1- \Vert T \Vert} +1 \right) \Vert  \nu \Vert_{L^{-\alpha/2, 2} (\Omega)} =  \left(   \frac{1 }{1- \Vert T \Vert}  \right) \Vert  \nu \Vert_{L^{-\alpha/2, 2} (\Omega)} .\]

We define $\mathcal{G}\nu$ by linearity when $\nu$ is a linear combination of positive measures in $L^{-\alpha/2, 2}(\Omega)$; then $\mathcal{G} \nu$ is defined q.e. (in fact whenever each term in the sum defining $G \nu$ is finite) and $u_0 = \mathcal{G} \nu \in L^{\alpha/2,2}_0 (\Omega)$ is quasicontinuous.  By Lemma \ref{Gbnded}, we still obtain (\ref{normu0est}), by the same steps, and the equation $u_0= G (u_0 \omega) + G \nu$ holds q.e. and as elements of $L^{\alpha/2,2}_0 (\Omega)$.

Since the linear combinations of positive measures are dense in $L^{-\alpha/2, 2}(\Omega)$, we can extend the map $\mathcal{G}$ to a bounded map (with the same bound) from $L^{-\alpha/2, 2}(\Omega)$ into $L^{\alpha/2, 2}_0 (\Omega)$. As for $G$, for $\mu \in L^{-\alpha/2, 2}(\Omega)$, we further define $\mathcal{G} \mu$ pointwise q.e. to be a quasicontinuous representative of its equivalence class in $L^{\alpha/2, 2}_0 (\Omega)$. To show that $u_0 \equiv \mathcal{G} \nu$ satisfies $u_0 = G(u_0 \omega) + G (\nu)$ for a general $\nu \in L^{-\alpha/2, 2}(\Omega)$, let $\nu_j$ be a sequence of linear combinations of positive measures converging to $\nu$ in $L^{-\alpha/2, 2}(\Omega)$.  Let $u_{0,j} = \mathcal{G} (\nu_j)$.  Then 
\begin{equation}\label{u0j}
 u_{0,j} = G ( u_{0,j} \omega) + G \nu_j ,  
\end{equation}
for each $j$.  By continuity of $G$ and $\mathcal{G}$, we have that $G\nu_j$ converges to $G \nu$, and $u_{0,j} = \mathcal{G} \nu_j $ converges to $\mathcal{G} \nu = u_0$ in $L^{\alpha/2, 2}_0 (\Omega)$.  By (\ref{embedd2}), $u_{0,j}$ converges to $u_0$ in $L^2 (\omega)$, since $u_{0,j} = \mathcal{G} \nu_j$ and $u_0 = \mathcal{G} \nu$ are quasicontinuous (without this convention we would need to replace them with quasicontinuous representatives at this point).  Then by (\ref{dualform}), $u_{0,j} \omega$ converges to $u_0 \omega$ in $L^{-\alpha/2, 2}(\Omega)$.  By the boundedness of $G$, then, $ G ( u_{0,j} \omega)$ converges to $G(u_0 \omega)$ in $L^{\alpha/2, 2}_0 (\Omega)$.  Hence taking the limit as $j \rightarrow \infty$ in (\ref{u0j}), we see that $u_0 = G(u_0 \omega) + G (\nu)$ holds in the sense of equality in  $L^{\alpha/2, 2}_0 (\Omega)$,  hence a.e., and therefore q.e. since both sides of the equation are quasicontinuous.

We now show that $u_0$ is the weak solution of (\ref{u0equation}).  We claim that for any $\mu \in L^{- \alpha/2, 2} (\Omega)$ and any $\varphi \in L^{\alpha/2, 2}_0 (\Omega)$,
\begin{equation} \label{partialweak}
 \int_{\R^n} (-\triangle)^{\alpha/4} G \mu  (-\triangle)^{\alpha/4} \varphi \, dx =  \langle \mu, \varphi \rangle.  
\end{equation}
First suppose $\mu \in L^{- \alpha/2, 2} (\Omega)$ is a finite positive measure on $\Omega$ and $\varphi \in C^{\infty}_0 (\Omega)$.  Let $\mu^{\, \prime} \in L^{- \alpha/2, 2} (\Omega)$ be as in (\ref{bal1})-(\ref{balayagedef}).  Then $I_{\alpha/2} \mu \in L^2 (\R^n)$, and hence $(-\triangle )^{\alpha/4} I_{\alpha} \mu = I_{\alpha/2 }\mu$ by Fourier transform, and similarly for $\mu^{\, \prime}$.  By (\ref{balayagedef}), 
\[  \int_{\R^n} (-\triangle)^{\alpha/4} G \mu  (-\triangle)^{\alpha/4} \varphi \, dx = \int_{\R^n}  I_{\alpha/2} (\mu -  \mu^{\, \prime}) (-\triangle)^{\alpha/4} \varphi \, dx \]
\[ = \int_{\R^n}  I_{\alpha/2} (-\triangle)^{\alpha/4}  \varphi \, d \mu - \int_{\R^n}  I_{\alpha/2} (-\triangle)^{\alpha/4}  \varphi \, d \mu^{\, \prime}\]
\[ = \int_{\R^n} \varphi \, d \mu -  \int_{\R^n} \varphi \, d \mu^{\, \prime} =  \int_{\R^n} \varphi \, d \mu = \langle \mu, \varphi \rangle,\] 
where the intermediate steps are justified via the Fourier transform, and $\int_{\R^n} \varphi \, d \mu^{\, \prime}=0$ because $\mu^{\, \prime}$ is supported in $\Omega^c$.  
Since 
\[ \left|\int_{\R^n} (-\triangle)^{\alpha/4} G \mu  (-\triangle)^{\alpha/4} \varphi \, dx  \right| \leq \Vert G\mu \Vert_{L^{\alpha/2,2}(\Omega)} \Vert \varphi \Vert_{L^{\alpha/2,2}(\Omega)} ,\] 
by the Cauchy-Schwarz inequality, we can extend (\ref{partialweak}) to all $\varphi \in L^{\alpha/2,2}(\Omega)$ by applying (\ref{partialweak}) to a sequence $\{\varphi_j\}_{j=1}^{\infty}$ of elements of $C^{\infty}_0 (\Omega)$ converging to $\varphi$ in $L^{\alpha/2,2}(\Omega)$ and letting $j \rightarrow \infty$.  

Next, suppose $\mu\in L^{- \alpha/2, 2} (\Omega)$ is a positive measure, not necessarily finite.  Define $\{\mu_j\}_{j=1}^{\infty}$ as in the proof of Lemma \ref{Gbnded}.  In that proof, we saw that each $\mu_j$ is finite, and $\mu_j \rightarrow \mu$ in the norm on $L^{- \alpha/2, 2} (\Omega)$.  Hence $G\mu_j$ converges to $G \mu$ in $L^{ \alpha/2, 2}_0 (\Omega)$.  Therefore, applying (\ref{partialweak}) to $\mu_j$ and taking the limit, we obtain (\ref{partialweak}) for $\mu$ and all $\varphi \in L^{\alpha/2,2}(\Omega)$.  This result then extends to linear combinations of positive measures in $L^{- \alpha/2, 2} (\Omega)$.  Then, by Deny's Theorem again, such linear combinations are dense in $L^{- \alpha/2, 2} (\Omega)$, so another passage to the limit implies (\ref{partialweak}) for all $\nu \in L^{- \alpha/2, 2} (\Omega)$ and all $\varphi \in L^{\alpha/2,2}(\Omega)$.

Now for $\nu \in L^{- \alpha/2, 2} (\Omega)$, we have that $u_0= \mathcal{G} \nu \in L^{\alpha/2, 2}_0 (\Omega)$ is quasicontinuous.  Then $u_0 \in L^2 (\omega)$, by Lemma \ref{Tbetaequiv} (iii).  Then by Lemma \ref{Tbetaequiv} (iv), $u_0 \omega \in L^{- \alpha/2,2}(\Omega)$.  We also have that $u_0 = G(u_0 \omega + \nu)$ as elements of $L^{\alpha/2, 2}_0 (\Omega)$. By (\ref{partialweak}),
\[  \int_{\R^n} (-\triangle)^{\alpha/4} u_0  (-\triangle)^{\alpha/4} \varphi \, dx  \]
\[ =  \int_{\R^n} (-\triangle)^{\alpha/4} G(u_0 \omega + \nu)  (-\triangle)^{\alpha/4} \varphi \, dx  =  \langle u_0 \omega + \nu, \varphi \rangle .\]
Since $u_0 \in L^2(\omega)$, we have $u_0^+, u_0^- \in L^2 (\omega)$.  By Lemma \ref{Tbetaequiv} (iv), we obtain $u_0^+ \omega, u_0^- \in L^{- \alpha/2, 2} (\Omega)$, hence $|u_0 \omega| \in L^{- \alpha/2, 2} (\Omega)$.  This allows us to conclude that 
\[ \langle u_0 \omega , \varphi \rangle = \int_{\Omega} u_0 \tilde{\varphi} \, d \omega  = \int_{\Omega} \widetilde{u_0} \tilde{\varphi} \, d \omega \]
for all $\varphi \in L^{\alpha/2, 2}_0 (\Omega)$, because $\widetilde{u_0} =u_0 $ since $u_0$ is quasicontinuous. Therefore $u_0$ is the weak solution of (\ref{u0equation}).

\vspace{0.1in}

(B) \,\, Now suppose that (\ref{intformulation}) has a non-negative solution $u$ for some non-trivial positive measure $\nu$ (i.e., $\nu (\Omega)>0$). Since $\nu$ is nontrivial, there exists a compact subset $K$ of $\Omega$ such that $\nu (K)>0$, hence by (\ref{Harnackconsequence}), $G \nu (x) \geq G\nu_K (x) >0$.  Since $Tu \geq 0$, we get that $u \geq G\nu >0$ on $\Omega$.  So $0 < u < \infty$ a.e. on $\Omega$ and satisfies $u = Tu + G\nu$, hence $Tu \leq u$ on $\Omega$.  Schur's Lemma for integral operators implies that $\Vert T \Vert \leq 1$.  

Now suppose also that $u=G(u \omega+ \nu) \in L^{\alpha/2, 2}_0 (\Omega)$. Then $\int_\Omega G (u \omega + \nu)  \, ( u\, d  \omega + d\nu) <\infty$, by Lemma \ref{Gbnded}. Consequently 
$\int_\Omega G \nu \, d \nu <\infty$.  Hence $\nu \in L^{-\alpha/2, 2} (\Omega)$, by Lemma \ref{Gbnded} again.
\end{proofof}

\bgp

\section{Proofs of Theorems \ref{nu=1thm}, \ref{u1thm}, and \ref{example}}

For the remainder of this paper, we assume that $\Omega$ is a bounded $C^{1,1}$ domain.  For such domains, 
the following estimates for $G$ are obtained in \cite{Ku} and \cite{CS}:
\begin{equation} \label{Gequivest}
G(x,y) \approx \frac{\delta(x)^{\alpha/2} \, \delta(y)^{\alpha/2}}{|x-y|^{n - \alpha} (|x-y| + \delta (x) + \delta(y))^{\alpha} },
\end{equation}
where ``$\approx$" means that the ratio of the left and right sides is bounded above and below by positive constants depending only on $\alpha$ and $\Omega$.  A useful consequence of (\ref{Gequivest}) is:
\begin{equation} \label{G1equiv}
G1 (x) \approx \delta (x)^{\alpha/2} .
\end{equation}
The lower bound follows because the quantities in the denominator of the right side of (\ref{Gequivest}) are bounded above on a bounded domain, so $G1 (x) \geq c \delta (x)^{\alpha/2} \int_{\Omega} \delta(y)^{\alpha/2} \, dy$.  The upper bound is similar: by (\ref{Gequivest}), we have $G (x,y) \leq c \delta(x)^{\alpha/2} |x-y|^{-n + \alpha/2}$;
then $\sup_{x \in \Omega} \int_{\Omega} |x-y|^{-n + \alpha/2} \, dy < \infty$ since $\Omega$ is bounded.

As for the classical Laplacian considered in \cite{FV}, our results are based on the estimates in \cite{FNV} for quasi-metric kernels.  We refer to \cite{FNV} (or the summary in \S 3 of \cite{FV}) for the definitions and details. The equivalence (\ref{Gequivest}) shows that for $m(x) = \delta (x)^{\alpha/2}$, the kernel $K(x,y) = \frac{G(x,y)}{m(x)m(y)}$ is a quasi-metric kernel.  Let $v_0 = \sum_{j=0}^{\infty} T^j m$.  Then Corollary 3.5 of \cite{FNV} states that there exists $c>0$ depending only on the quasimetric constant of $K$ such that 
\begin{equation}\label{v0lowerbnd}
v_0 \geq m e^{c Tm /m}.
\end{equation}
If we assume in addition that $\Vert T \Vert <1$, the same result states that there exists $C>0$ depending only on $\Vert T \Vert$ and the quasi-metric constant of $K$ such that 
\begin{equation} \label{v0upperbnd}
v_0 \leq m e^{C Tm /m}.
\end{equation}

\begin{proofof}{Theorem \ref{nu=1thm}}. 
(i) First suppose that $\Vert T \Vert <1$. Since $\nu \in L^{-\alpha/2, 2}(\Omega)$ (see the remarks in the Introduction before the statement of Theorem \ref{nu=1thm}), by Proposition \ref{u0thm} we have $u_0 = \mathcal{G} \nu \in L^{\alpha/2,2}_0 (\Omega)$ and $u_0$ is a solution of $u= G(u \omega) + G1$.  Because of (\ref{G1equiv}) we obtain $u_0 = \sum_{j=0}^{\infty} T^j G1  \approx \sum_{j=0}^{\infty} T^j m = v_0$, so the estimate (\ref{u0upperestimate}) follows from (\ref{v0upperbnd}). 

(ii) As in \cite{FV}, p. 1405, $u_0$ is the minimal positive solution of (\ref{intformulation}).  Hence $u \geq u_0$.  Since $u_0 \approx v_0$,  (\ref{ulowerestimate}) follows from (\ref{v0lowerbnd}).   

\end{proofof}

Turning to equation (\ref{u1equation}), we first recall that for $v=u-1$, equation (\ref{u1equation}) becomes $(-\triangle)^{\alpha/2} v = \omega v + \omega$ on $\Omega$ with $v=0$ on $\Omega^c$, which is equation (\ref{u0equation}) with $\nu$ replaced by $\omega$.  Therefore if we assume $\omega \in L^{-\alpha/2,2}(\Omega)$ and $\Vert T \Vert <1$, we obtain $v_1 = \mathcal{G} \omega \in L^{\alpha/2,2}_0(\Omega)$, and $v_1$ is the unique weak solution guaranteed by the Lax-Milgram Theorem. However, the Lax-Milgram Theorem applies only when $\omega \in L^{-\alpha/2,2}(\Omega)$, whereas the integral formulation (\ref{intformulationu1}) allows us to consider more general $\omega$.  In Remark \ref{last} we give an example where the integral equation holds a.e., but $\omega \not\in L^{-\alpha/2, 2} (\Omega)$. 

As in the case $\alpha =2$ in \cite{FV}, the functions $u_0 = \mathcal{G} 1$ and $u_1 = 1 + \mathcal{G} \omega$ are related by the identity
\begin{equation} \label{u0u1relation}
\int_{\Omega} u_1 \, dx = \int_{\Omega} 1 \, dx + \int_{\Omega} \int_{\Omega} \mathcal{G} (x,y) \, d \omega (y) \, dx = |\Omega| + \int_{\Omega} u_0 \, d \omega, 
\end{equation}
using the symmetry of $\mathcal{G}$ and Fubini's theorem.  

We make some remarks about the Poisson kernel.  For bounded domains with the outer cone property (in particular, bounded Lipschitz domains; see \cite{BBK}, pp. 16-17, or \cite{CS}, p. 468) the Poisson kernel of order $\alpha$ satisfies $\int_{\Omega^c} P(x,y) \, dy =1$ for all $x \in \Omega$ and can be written
\[  P(x,y) = P^{(\alpha)} (x,y)= A_{\alpha, n} \int_{\Omega} \frac{ G(x,z)}{|y-z|^{n + \alpha}} \, dz , \]
for $x \in \Omega$ and $y \in \Omega^c$,  where $A_{\alpha, n}$ is a constant.  For $x \in \Omega$, define 
\begin{equation} \label{defphi}
  \phi (x) =  A_{\alpha, n}  \int_{\Omega^c} \frac{1}{|x-z|^{n + \alpha} } \, dz. 
\end{equation}
Hence  
\begin{equation} \label{poissonid}
G \phi (x) = \int_{\Omega^c} P(x,y) \, dy =1 \,\, \mbox{for all} \,\, x \in \Omega.
\end{equation}
 We note that there exist positive constants $c(\alpha, \Omega)$ and $C(\alpha, n)$ such that 
\begin{equation} \label{phiequiv}
  c(\alpha, \Omega) \frac{1}{\delta (x)^{\alpha}} \leq \phi (x) \leq c_n A_{\alpha, n} \alpha^{-1} \frac{1}{\delta (x)^{\alpha}}, \,\, \mbox{for all} \,\, x \in \Omega,  
\end{equation}
where $\delta (x)$ is the distance from $x$ to $\Omega^c$.  The upper bound in (\ref{phiequiv}) is elementary, whereas the lower bound follows because $C^{1,1}$ domains (in fact NTA domains) have the property that there are constants $c>0$ and $r_0>0$ such that $|B(y,r) \cap \Omega^c| \geq c |B(y,r)|$, for all $y \in \partial \Omega$ and $0<r< r_0$.  

We will also use the well-known equivalence (\cite{CS}, Theorem 1.5)
\begin{equation}\label{Poissonest}
P(x,y) \approx \frac{\delta(x)^{\alpha/2}}{\delta(z)^{\alpha/2} (1+ \delta(z))^{\alpha/2} |x-z|^n},
\end{equation}
where here $\delta (z) = $ dist $(z, \partial \Omega)$, with equivalence constants independent of $x$ and $y$.

\vspace{0.1in}

\begin{proofof}{Theorem \ref{u1thm}}. First suppose $0< \alpha <1$.  By (\ref{poissonid}), $G \phi = \chi_{\Omega}$.  By (\ref{phiequiv}),we have 
\[ \int_{\Omega} G \phi (x) \phi(x) \, dx  = \int_\Omega \phi(x) \, dx \le  C(\alpha, n)  \int_\Omega \delta (x)^{-\alpha}\, dx <+\infty \]
for $0<\alpha<1$, 
for a broad class of domains $\Omega$ (e.g., Ahlfors regular domains; in particular, bounded Lipschitz domains). By Lemma \ref{Gbnded},  
$\chi_\Omega\in L^{\alpha/2, 2}_0 (\Omega)$ for all $0<\alpha<1$. Then by (\ref{embedd2}), $\chi_{\Omega} \in L^2 (\omega)$, or $\omega (\Omega) < \infty$; i.e., $\omega$ is a finite measure.  By Theorem \ref{u0thm} and (\ref{embedd2}), $u_0 \in L^2 (\omega)$, so by the finiteness of $\omega$, we have $u_0 \in L^1 (\omega)$.   Thus by (\ref{u0u1relation}), $u_1 \in L^1 (\Omega, dx)$. 
  
  Notice that $G \omega= T(\chi_{\Omega})$, hence
\[ \int_\Omega G \omega \, d\omega \leq \omega(\Omega)^{1/2} \Vert T(\chi_{\Omega}) \Vert_{L^2(\omega)} \leq \beta^2 \omega(\Omega) < \infty.  \]  
Therefore $\omega \in L^{- \alpha/2, 2} (\Omega)$, by Lemma \ref{Gbnded}.  By Theorem \ref{u0thm}, 
$u_1-1 = \mathcal{G} \omega \in L^{\alpha/2, 2}_0 (\Omega)$.

Now suppose $1 \leq \alpha < 2$.  By (\ref{poissonid}), for any non-negative Borel measure $\nu$ on $\Omega$, we have 
\[  \nu(\Omega) = \int_{\Omega} G \nu (x) \, \phi (x) \, dx ,  \]
by Fubini's theorem.  Applying this fact with $d \nu = u_0 d \omega$, where $u_0$ is defined by (\ref{defu0}) and satisfies $u_0 = G(u_0 \omega) + G1$, yields 
\begin{equation} \label{identityu0omega}
  \int_{\Omega} u_0  \, d \omega = \int_{\Omega} G(u_0 \, \omega) \phi \, dx = \int_{\Omega} u_0 \phi \, dx - \int_{\Omega} G1 \cdot \phi \, dx . 
\end{equation}
By (\ref{G1equiv}) and (\ref{phiequiv}), $G1 \cdot \phi \approx \delta^{-\alpha/2}$, and hence $\int_{\Omega} G1 \cdot \phi \, dx < \infty$.  Our goal is to show that $\int_{\Omega} u_0 \phi \, dx  < \infty$ for $\Vert T \Vert$ sufficiently small. 

Recall that if $\Vert T \Vert <1$, then for $m = \delta^{\alpha/2}$, we have
\[ u_0  \leq C_1 m e^{C Tm /m} , \]
by Theorem \ref{nu=1thm}, where $C= C(\Omega, \alpha, \Vert T \Vert)$.  Choose and fix $p > \frac{2+\alpha}{2 - \alpha}$, which guarantees that $\frac{\alpha (p+1)}{2 (p-1)} <1$.  Let $C_2 = C(\Omega, \alpha, 1/p)$; that is, $C_2$ is the constant $C$ when $\Vert T \Vert = 1/p$.  Let $c$ be the constant from (\ref{ulowerestimate}); note that $c \leq C_2$ (e.g., by (1.4) in \cite{FNV}).  Define
\[   \gamma = \frac{c}{C_2 p}  .\]
Note that $\gamma \leq 1/p$.  Now suppose $\Vert T \Vert < \gamma$.  Then $u_0 \leq C_1 m e^{C_2 Tm/m}$. Hence, by H\"{o}lder's inequality (using (\ref{phiequiv}))
\[  \int_{\Omega} u_0 \phi \, dx \leq C(\alpha, n) \int_{\Omega} \frac{u_0}{m^2} \, dx \leq C_1 C(\alpha, n) \int_{\Omega} \frac{1}{m} e^{C_2 Tm/m} \, dx \]
\[ \leq C_1 C(\alpha, n) \left( \int_{\Omega} m^{-\frac{p+1}{p-1}} \, dx \right)^{1- \frac{1}{p}} \left(\int_{\Omega} m e^{C_2 p Tm/m} \, dx  \right)^{1/p} . \]
Since $m^{-\frac{p+1}{p-1}} = \delta^{\frac{\alpha (p+1)}{2 (p-1)}}$, the first integral on the previous line is finite. To show that the second integral is finite, let $\omega_1  = \gamma^{-1} \omega$.  We apply Theorem \ref{nu=1thm} with $\omega_1$ in place of $\omega$, but with $G$ and $m$ unchanged.  
Define $T_1 = \gamma^{-1}T$; note that $T_1 f (x) = G(\omega_1 f)$, and  
\[ \Vert T_1 \Vert_{L^2 (\omega_1) \rightarrow L^2 (\omega_1)} =  \Vert T_1 \Vert_{L^2 (\omega) \rightarrow L^2 (\omega)} = \gamma^{-1} \Vert T \Vert_{L^2 (\omega) \rightarrow L^2 (\omega)} <  1 .\]
Define $u_0^* = \sum_{j=0}^{\infty} T_1^j G1 $.  By Theorem \ref{nu=1thm},  $c_1 m e^{c T_1 m/m } \leq  u_0^*$ and $ u_0^* \in L^{\alpha/2, 2}_0 (\Omega) \subseteq L^{p^*} (\Omega,dx) \subseteq$ $L^1 (\Omega, dx)$ (since $\Omega$ is bounded), for $p^* = 2n/(n- \alpha)$.  But $C_2 pTm = c  T_1 m$, so
\[ \int_{\Omega} m e^{C_2 p Tm/m} \, dx  = \int_{\Omega} m  e^{c T_1 m/m } \, dx \leq c_1^{-1} \int_{\Omega}  u_0^* \, dx < \infty.  \]
We have shown $\int_{\Omega} u_0 \phi \, dx <\infty$, hence $u_0 \in L^1 (\omega)$, by (\ref{identityu0omega}).  By (\ref{u0u1relation}), we have $u_1 \in L^1 (\Omega, dx)$.

We now turn to the pointwise bounds (\ref{u1pointwiseupper}) and (\ref{u1pointwiselower}).  Their proofs are similar to the proofs of (1.12) and (1.14) in \cite{FV}.  By the same argument as on p. 1413 of \cite{FV}, using (\ref{poissonid}) we have
\begin{equation} \label{newu1iden}
 u_1 (x) = \int_{\Omega^c} \sum_{j=0}^{\infty} T^j (P(\cdot, z))(x)  \, dz.  
 \end{equation} 
Define the quasi-metric 
\[ d(x,y) = |x-y|^{n-\alpha} [ |x-y|^2 + \delta(x)^2 + \delta(y)^2 ]^{\alpha/2} , \,\,\, x,y \in \R^n . \]
Note that for $x \in \Omega$ and $z \in \Omega^c$, we have $d(x,z) \approx |x-z|^n$, because $\delta (x), \delta(z) \leq |x-z|$.
Momentarily fixing $z \in \overline{\Omega}^{\, c}$, we let $m(x) = P(x,z)$.  By (\ref{Gequivest}), it follows that 
\[ K(x,y) \equiv \frac{G(x,y)}{m(x) m(y)} \approx c(z)  \frac{d(x,z) d(y,z)}{d(x,y)},  \]
where $c(z) = \delta(z)^{\alpha} (1+ \delta(z))^{\alpha}$.  From a lemma due to Hansen and Netuka (\cite{HN}), quoted as Lemma 3.4 in \cite{FV}, it follows that $K(x,y)$ is a quasi-metric kernel on $\Omega$ with quasi-metric constant independent of $z$.  By Corollary 3.5 in \cite{FNV} (essentially (\ref{v0upperbnd})), 
\[   \sum_{j=0}^{\infty} T^j (P(\cdot, z))(x)  \leq C_3 P(x,z) e^{C_4 \int_{\Omega} G(x,y) \frac{P(y,z)}{P(x,z)}\, d \omega (y)  }   ,   \]
with constants independent of $z$.  Using (\ref{newu1iden}), then, we obtain (\ref{u1pointwiseupper}). We also have the lower estimate (\ref{u1pointwiselower}) for $u_1$, again by Corollary 3.5 in \cite{FNV}.  Since $u_1$ is the minimal positive solution of $u=1+G(u\omega)$, we obtain (\ref{u1pointwiselower}) for $u$ as well.

\end{proofof}

\vspace{0.1in}

\begin{proofof}{Theorem \ref{example}}.  Let $d \omega = \phi \, dx$, where $\phi $ is defined by (\ref{defphi}).  By (\ref{poissonid}), 
\[  T1 (x) = G \phi (x) =1 \, \, \mbox{for all} \,\, x \in \Omega.  \] 
Hence $T^j 1 =1$ for all $j=1, 2, \dots$, so $u_1 = 1 + \sum_{j=1}^{\infty} T^j G \phi \equiv + \infty$ on $\Omega$. 

It follows (see e.g., \cite{DPV}, Proposition 3.4 and the proof of Lemma 5.1) from Plancherel's theorem that for, say, $u \in C^{\infty}_0 (\Omega)$, 
\[ \Vert u \Vert_{L^{\alpha/2, 2}_0 (\Omega)}^2 = \Vert u \Vert_{L^{\alpha/2, 2}_0 (\R^n)}^2 = \frac{A_{\alpha, n}}{2} \int_{\R^n} \int_{\R^n} \frac{|u(x)-u(y)|^2}{|x-y|^{n + \alpha}} \, dx \, dy \]
\[  =  \frac{A_{\alpha, n}}{2} \int_{\Omega} \int_{\Omega} \frac{|u(x)-u(y)|^2}{|x-y|^{n + \alpha}} \, dx \, dy + \int_{\Omega} u^2(x) \phi (x) \, dx ,  \]
where $A_{\alpha, n}$ is the constant from (\ref{phiequiv}).
The following version of Hardy's inequality holds for bounded Lipschitz domains 
if $1 < \alpha <2$ (\cite{Dy}, Theorem 1.1):  
\[ \int_{\Omega}u^2  \delta^{-\alpha} \, dx \leq C_1 (\alpha,\Omega) \int_{\Omega} \int_{\Omega} \frac{|u(x)-u(y)|^2}{|x-y|^{n + \alpha}} \, dx \, dy, \,\, \mbox{for all} \,\, u \in C^{\infty}_0 (\Omega). \]
Recalling (\ref{phiequiv}), we obtain 
\[  \int_{\Omega}u^2  \phi  \, dx \leq \frac{c_n A_{\alpha, n}}{\alpha} \int_{\Omega}u^2  \delta^{-\alpha} \, dx  \leq \frac{c_n  C_1 (\alpha, \Omega)A_{\alpha, n}}{\alpha} \int_{\Omega} \int_{\Omega} \frac{|u(x)-u(y)|^2}{|x-y|^{n + \alpha}} \, dx \, dy.  \]
Therefore
\[ \left(1 + \frac{\alpha}{2c_n C_1 (\alpha, \Omega) } \right) \int_{\Omega} u^2 \phi \, dx \leq \Vert u \Vert_{L^{\alpha/2, 2}_0 (\Omega)}^2 , \,\, \mbox{for all} \,\, u \in C^{\infty}_0 (\Omega). \]
Hence, for $d\omega = \phi \, dx$, (\ref{embedd}) holds with $\beta=\left(1 + \frac{\alpha}{2c_nC_1 (\alpha, \Omega)}\right)^{-1/2}<1$. By Lemma \ref{Tbetaequiv}, we have $\Vert T \Vert \leq \beta$.
\end{proofof}

By \cite{FMT}, p. 115, for convex domains $\Omega$ the constant $C_1(\alpha, \Omega)$ depends only on $\alpha$ and the dimension $n$, and the value 
\[ C_1(\alpha, \Omega) = \frac{\alpha\Gamma (\frac{n + \alpha}{2})}{2^{2-\alpha}\pi^{\frac{n-2}{2}} \Gamma(1 - \frac{\alpha}{2})\Gamma^2 (\frac{\alpha +1}{2})}  \] 
is sharp.  In that case, as $\alpha \rightarrow 2^-$, we have $C_1 (\alpha, \Omega) \rightarrow 0$, so in the above proof, $\beta \rightarrow 0$.  Consequently, the value of $\gamma$ in Theorem \ref{u1thm} must converge to $0$ as $\alpha \rightarrow 2^-$.

\begin{Rem} \label{last} We observe that, for  $d \omega = \gamma \, \phi \, dx$, where $\phi $ is defined by (\ref{defphi}) as above and $\gamma\in(0,1)$, 
 we  have  $u_1(x)=1/(1-\gamma)$ in $\Omega$,  and $u_1(x)=1$ in $\Omega^c$. Then $u_1= G(u_1 \, \omega)+1$ in $\R^n$, and  $(-\Delta)^{\alpha/2} u_1 = \omega \, u_1$ in $D'(\Omega)$. 
 However, in contrast to the case $0<\alpha<1$, 
we have $\omega \not\in L^{-\alpha/2, 2}(\Omega)$ and $u_1-1 =\mathcal{G} \omega\not\in L^{\alpha/2, 2}_0 (\Omega)$ for all $1\le \alpha <2$;  in fact,  
 $u_1$ obviously does not have a quasi-continuous representative in $\R^n$ since ${\rm cap}_\alpha(\partial \Omega)>0$  in this case.
\end{Rem}

\end{document}